%% file: SCsubring.tex
\documentclass[11pt,leqno]{amsart}

\usepackage{latexsym}
\usepackage{amssymb}
\usepackage{amsmath}
\usepackage{amsthm}
\usepackage{verbatim}
\usepackage{pdfsync}
\usepackage[margin=1.1in]{geometry}
\usepackage{tikz}
\usepackage[curve,matrix,arrow,frame,tips]{xy}

\numberwithin{equation}{section}  

\usepackage{color}
\usepackage{graphicx}

\makeatletter
\newcommand*{\bigcdot}{}
\DeclareRobustCommand*{\bigcdot}{%
  \mathbin{\mathpalette\bigcdot@{}}%
}
\newcommand*{\bigcdot@scalefactor}{.5}
\newcommand*{\bigcdot@widthfactor}{1.15}
\newcommand*{\bigcdot@}[2]{%
  \sbox0{$#1\vcenter{}$}
  \sbox2{$#1\cdot\m@th$}%
  \hbox to \bigcdot@widthfactor\wd2{%
    \hfil
    \raise\ht0\hbox{%
      \scalebox{\bigcdot@scalefactor}{%
        \lower\ht0\hbox{$#1\bullet\m@th$}%
      }%
    }%
    \hfil
  }%
}
\makeatother

\begin{document}

\newtheorem*{theob}{Theorem B}



\newcommand{\zxz}[4]{\begin{pmatrix} #1 & #2 \\ #3 & #4 \end{pmatrix}}
\newcommand{\abcd}{\zxz{a}{b}{c}{d}}
\newcommand{\kzxz}[4]{\left(\begin{smallmatrix} #1 & #2 \\ #3 &
#4\end{smallmatrix}\right) }
\newcommand{\kabcd}{\kzxz{a}{b}{c}{d}}

\input KRY.macros

\newcommand{\OO}{\text{\rm O}}
\newcommand{\UU}{\text{\rm U}}

\newcommand{\OK}{O_{\smallkay}}
\newcommand{\DI}{\mathcal D^{-1}}

\newcommand{\pre}{\text{\rm pre}}

\newcommand{\Bor}{\text{\rm Bor}}
\newcommand{\Rel}{\text{\rm Rel}}
\newcommand{\rel}{\text{\rm rel}}
\newcommand{\Res}{\text{\rm Res}}
\newcommand{\TG}{\widetilde{G}}

\newcommand{\OL}{O_{\Lambda}}
\newcommand{\OLB}{O_{\Lambda,B}}

\newcommand{\p}{\varpi}

\newcommand{\cutter}{\vskip .1in\hrule\vskip .1in}

\parindent=0pt
\parskip=10pt
\baselineskip=16pt

\newcommand{\PP}{\mathcal P}
\renewcommand{\OO}{\mathcal O}
\newcommand{\BB}{\mathbb B}
\newcommand{\OBB}{O_{\BB}}
\newcommand{\Max}{\text{\rm Max}}
\newcommand{\Opt}{\text{\rm Opt}}
\newcommand{\OH}{O_H}

\newcommand{\phhat}{\widehat{\phi}}
\newcommand{\thetahat}{\widehat{\theta}}

\newcommand{\lbold}{\text{\boldmath$\l$\unboldmath}}
\newcommand{\abold}{\text{\boldmath$a$\unboldmath}}
\newcommand{\cbold}{\text{\boldmath$c$\unboldmath}}
\newcommand{\ebold}{\text{\boldmath$e$\unboldmath}}
\newcommand{\aabold}{\text{\boldmath$\a$\unboldmath}}
\newcommand{\gbold}{\text{\boldmath$g$\unboldmath}}
\newcommand{\hbold}{\text{\boldmath$h$\unboldmath}}
\newcommand{\obold}{\text{\boldmath$\o$\unboldmath}}
\newcommand{\fbold}{\text{\boldmath$f$\unboldmath}\!}
\newcommand{\rbold}{\text{\boldmath$r$\unboldmath}}
\newcommand{\ffbold}{\und{\fbold}}
\newcommand{\sbold}{\text{\boldmath$\s$\unboldmath}}
\newcommand{\tbold}{\text{\boldmath$t$\unboldmath}}
\newcommand{\qbold}{\text{\boldmath$q$\unboldmath}}
\newcommand{\gammabold}{\text{\boldmath$\gamma$\unboldmath}}

\newcommand{\zbold}{\text{\boldmath$z$\unboldmath}}
\newcommand{\mbold}{\text{\boldmath$m$\unboldmath}}
\newcommand{\chibold}{\text{\boldmath$\chi$\unboldmath}}

\newcommand{\boldCC}{\text{\boldmath$\mathcal C$\unboldmath}}
\newcommand{\kbold}{\text{\boldmath$k$\unboldmath}}
\newcommand{\fbolds}{\fbold\,}

\newcommand{\deltaBB}{\delta_{\BB}}
\newcommand{\kappaBB}{\kappa_{\BB}}
\newcommand{\aboldBB}{\abold_{\BB}}
\newcommand{\lboldBB}{\lbold_{\BB}}
\newcommand{\gboldBB}{\gbold_{\BB}}
\newcommand{\bbold}{\text{\boldmath$\b$\unboldmath}}

\newcommand{\phbold}{\text{\boldmath$\ph$\unboldmath}}

\newcommand{\fff}{\phi}
\newcommand{\spp}{\text{\rm sp}}

\newcommand{\pob}{\mathfrak p_{\bold o}}
\newcommand{\kob}{\mathfrak k_{\bold o}}
\newcommand{\gob}{\mathfrak g_{\bold o}}
\newcommand{\pobp}{\mathfrak p_{\bold o +}}
\newcommand{\pobm}{\mathfrak p_{\bold o -}}

\newcommand{\mmm}{\mathfrak m}


\newcommand{\bb}{\frak b}

\newcommand{\bbbold}{\text{\boldmath$b$\unboldmath}}

\renewcommand{\ll}{\,\frak l}
\newcommand{\uC}{\underline{\Cal C}}
\newcommand{\uZZ}{\underline{\ZZ}}
\newcommand{\B}{\mathbb B}
\newcommand{\CL}{\text{\rm Cl}}

\newcommand{\pp}{\frak p}

\newcommand{\OKp}{O_{\smallkay,p}}

\renewcommand{\top}{\text{\rm top}}

\newcommand{\bF}{\bar{\mathbb F}_p}

\newcommand{\beq}{\begin{equation}}
\newcommand{\eeq}{\end{equation}}


\newcommand{\Dl}{\Delta(\l)}
\newcommand{\mm}{{\bold m}}

\newcommand{\FD}{\text{\rm FD}}
\newcommand{\LDS}{\text{\rm LDS}}

\newcommand{\dcM}{\dot{\Cal M}}
\newcommand{\bpm}{\begin{pmatrix}}
\newcommand{\epm}{\end{pmatrix}}

\newcommand{\GW}{\text{\rm GW}}

\newcommand{\uk}{\bold k}
\newcommand{\uo}{\text{\boldmath$\o$\unboldmath}}

\newcommand{\ww}{\bold w}
\newcommand{\yy}{\bold y}

\newcommand{\Om}{\text{\boldmath$\Omega$\unboldmath}}

\newcommand{\phnat}{{}^\nat\phi}
\newcommand{\om} {\text{\boldmath$\o$\unboldmath}}
\newcommand{\omnat}{\breve\om}

\newcommand{\bino}[2]{{#1\choose{#2}}}

\newcommand{\kk}{\frak k}

\newcommand{\gw}[2]{\langle\langle #1,#2 \rangle\rangle}
\newcommand{\gwrm}{\text{\rm gw}}

\newcommand{\cpar}{\sbold}


\newcommand{\N}{\Cal N}
\newcommand{\GG}{\Cal G}
\newcommand{\GGu}{\und{\GG}}
\newcommand{\LGG}{\Cal L \Cal G}

\newcommand{\car}{\text{\rm char}}

\newcommand{\Con}{\text{\rm Con}}

\newcommand{\h}{\mathfrak h}
\newcommand{\prz}{{\bf prz}}
\newcommand{\hc}{\text{\rm hc}}

\renewcommand{\z}{\mathfrak z}     
\newcommand{\x}{\und{x}}
\newcommand{\har}{\text{\rm har}}

\newcommand{\qq}{\qbold}
\newcommand{\qqq}{\mathfrak q}
\newcommand{\Skew}{\text{\rm Skew}}

\newcommand{\Dp}{D^+}
\newcommand{\Djp}{D^{(j),+}}
\newcommand{\ch}{\text{\rm CH}}
\renewcommand{\ra}{\rightarrow}
\newcommand{\wtg}{\widetilde{G'}}    

\newcommand{\Sb}{\mathcal S^{\bigcdot}}

\newcommand{\SC}{\text{\rm SC}}
\newcommand{\Ze}{\mathfrak Z^\nat}
\newcommand{\Zer}{\mathfrak Z}
\newcommand{\bH}{\breve{\mathcal H}}
\newcommand{\bcbold}{\cbold}
\newcommand{\btriv}{\triv}
\newcommand{\Rep}{\text{\rm Rep}}
\newcommand{\Repp}{\Rep_+}
\newcommand{\Vp}{V_+}
\newcommand{\Gp}{G_+}
\newcommand{\HSC}{H}
\newcommand{\ndeg}{\deg^{\text{\rm tot}}}
\newcommand{\trivn}{\triv^\nat}
\newcommand{\cboldn}{\cbold^\nat}

\newcommand{\now}{\count0=\time 
\divide\count0 by 60
\count1=\count0
\multiply\count1 by 60
\count2= \time
\advance\count2 by -\count1
\the\count0:\the\count2}


\author{ Stephen S. Kudla}

\title{On the subring of special cycles}

\maketitle

\section{Introduction}\label{section1}

In this paper, we investigate the structure of the ring generated by the cohomology classes of special cycles 
in orthogonal Shimura varieties over totally real fields. 
For a quadratic space $V$ of dimension $m+2$ over a totally real field $F$ of degree $d$ over $\Q$ with signature
\beq\label{def-sig-V}
\sig(V) = (m,2)^{d_+}\times (m+2,0)^{d-d_+},
\eeq
the special cycles and the subring they generate in Chow groups and in cohomology of the associated Shimura variety $\text{\rm Sh}(V)$
were considered in  \cite{kudla.rem-gen}. The case where $d_+=1$ was previously studied in \cite{K.duke}. 
In the present paper, we consider the subring of cohomology, equipping it with
an inner product given by the cup product of classes 
followed by a degree map on the top cohomology. In particular, classes not in complementary degrees pair to zero. 
The reduced ring of special cycles $SC^\bullet(V)$ is obtained by taking the quotient of the special cycle ring by the radical of this pairing.  
In order to work with compact quotients, we assume that $d_+<d$ or that $V$ is anisotropic.  We prove two main results. 

First, as a consequence of the Siegel-Weil formula, we show that the inner product of elements of $SC^\bullet(V)$ is determined by 
the Fourier coefficients of pullbacks  
of Hilbert-Siegel-Eisenstein series to products of Hilbert-Siegel subspaces under embeddings
\beq\label{diag-2}
\H_{n_1}^d\times \H_{n_2}^d\ \lra \ \H_m^d,\quad  [\tau_1,\tau_2]\mapsto\bpm \tau_1&{}\\{}&\tau_2\epm \qquad n_1+n_2=m.
\eeq
Moreover, 
the products in the ring $SC^\bullet(V)$ are completely determined by the Fourier coefficients of pullbacks  
of Hilbert-Siegel-Eisenstein series to triple products of Hilbert-Siegel subspaces   
respect to embeddings
\beq\label{diag-3}
\H_{n_1}^d\times \H_{n_2}^d\times \H_{n_3}^d \lra \H_m^d, \quad [\tau_1,\tau_2,\tau_3] \mapsto \bpm \tau_1&{}&{}\\{}&\tau_2&{}\\
{}&{}&\tau_3\epm,\qquad n_1+n_2+n_3 =m.
\eeq
Here $\H_n$ denotes the Siegel upper half space of genus $n$.  Precise statements are given in Theorem~A in section 2. 

Second, as a consequence of the `matching principle of \cite{K.Bints}, we show that for quadratic spaces $V$ and $V'$ of dimension $m+2$ over $F$
that are isomorphic at all finite places, the reduced special cycle rings are isometrically isomorphic, $SC^\bullet(V) \simeq SC^\bullet(V')$, Theorem~B in Section 1. Note that the associated Shimura 
varieties may have different dimensions, since $d_+(V)$ and $d_+(V')$ need not be equal, although they have the same parity. 
Since the special cycles occur in codimensions $nd_+$, the isomorphism can involve a shift in degrees as well. 

Finally, when $d_+(V)$ is even, there is an associated totally positive definite space $\Vp$ such that $V$ and $\Vp$ are isomorphic at all finite 
places. This space is unique up to isometry. In Section~\ref{section4}, we give a combinatorial construction, involving representation numbers, of a graded ring $\SC^\bullet(\Vp)$ of `special cycles' associated to $\Vp$
and show that there is a comparison isomorphism $\SC^\bullet(V) \isoarrow \SC^\bullet(\Vp)$.  In particular, the cohomological 
special cycle ring of the Shimura variety $\text{\rm Sh}(V)$ for $d_+(V)$ even has a purely combinatorial description. 

Our proof of these comparison results is indirect and depends on the Siegel-Weil formula, which might be regarded as a weak type of relative trace formula. 
It would be of interest to find a more direct geometric proof.

It should be noted that we work on orthogonal Shimura varieties with many connected components and with ad\`elic weighted special cycles 
on them. It seems possible that our results can be refined to cover the cohomological special cycle rings of the individual components of these varieties, 
but this would require a `twisted' variant of the Siegel-Weil formula involving the automorphic characters of special orthogonal 
groups obtained as a composition of the spinor norm with quadratic Hecke characters of $F$.  A few hints at such a formula occur in the literature, 
cf. \cite{snitz} and \cite{gan-snitz}.  For example, the result of \cite{snitz} is used in \cite{KRYbook} to isolate the degrees of $0$-cycles 
on individual connected components of Shimura curves over $\Q$. 

Finally, it should be straightforward to extend the results of this paper and of \cite{kudla.rem-gen} to the case of unitary groups of 
signature $(m,1)^{d_+}\times (m+1,0)^{d-d_+}$. The only exception is the modularity of the Chow group valued generating series, proved in 
\cite{kudla.rem-gen} to be a consequence of the Bloch-Beilinson conjecture, since the proof there depends on a combination of the embedding trick 
with a vanishing result for odd Betti numbers in relatively small degree. This vanishing result is a peculiarity of the group $\SO(m,2)$, 
and is not available for the unitary case. 

In we next two sections of the introduction, we give a more detailed description of our results. In Section~\ref{section1.1} we review the notation and results of \cite{kudla.rem-gen}.
In Section~\ref{subsec-1.2}, we give more precise statements.    

\subsection{Background}\label{section1.1}
For a totally real field $F$ of degree $d$, let $\Sigma=\{\s\}$ be the set of embeddings of $F$ into $\R$.  
For $(V,Q)$ a quadratic space of dimension $m+2$ over $F$ and $\s\in \Sigma$, let $V_\s=V\tt_{F,\s}\R$, and suppose that the signature 
of $V$ is given by 
\beq\label{def-sigV}
\sig(V_\s)= \begin{cases} (m,2)&\text{if $\s\in \Sigma_+(V)$,}\\
\nass
(m+2,0)&\text{otherwise,}
\end{cases}
\eeq
for a subset $\Sigma_+=\Sigma_+(V)$ of $\Sigma$ with $|\Sigma_+|=d_+$. 
We assume that $1\le d_+<d$ or that $V$ is anisotropic, leaving aside the problem of extending our results to the non-compact case. 

Let $G= R_{F/\Q}\GSpin(V)$ and let 
\beq\label{DVS}
D = \prod_{\s\in \Sigma_+} D^{(\s)}, \qquad D^{(\s)} = \{ \, z \in \text{\rm Gr}_2^o(V_\s) \mid \ Q\mid_z <0\,\},
\eeq
where $\text{\rm Gr}_2^o(V_\s)$ is the Grassmannian of oriented $2$-planes in $V_\s$. For a compact open subgroup 
$K \subset G(\A_f)$, let 
$$S_K = G(\Q)\back D\times G(\A_f)/K,$$
and recall that $S_K$ is isomorphic to the set of complex points of a projective variety which is smooth if $K$ is neat. The canonical model is defined over a reflex field determined by $\Sigma_+$, but 
we will not need this for the moment. 
The Chow groups and (Betti) cohomology groups (with complex coefficients) of these varieties, as $K$ varies, form an inverse system, and we define
$$
\ch^{k}(S)= \varinjlim\limits_K\, \ch^{k}(S_K)\qquad\text{and}\qquad
H^{k}(S)= \varinjlim\limits_K\, H^{k}(S_K), 
$$
and graded rings
$$\ch^{\bullet}(S) = \bigoplus_{k=0}^{md_+}\ch^k(S), \qquad\text{and}\qquad H^\bullet(S) = \bigoplus_{k=0}^{2m d_+} H^{k}(S)$$
under intersection product and cup product respectively. 
The group $G(\A_f)$ acts naturally on these rings. Of course, we will only be concerned with a subring of classes of Hodge type $(p,p)$. 

As explained in \cite{K.duke}, p.45, we have 
$$\pi_0(S_K) \simeq G(\Q)_+\back G(\A_f)/K \simeq F^\times_{\A_f}/F^\times_+ \nu(K),$$
where $\nu:G\ra R_{F/\Q}\G_m$ is the spinor norm and $F^\times_+$ is the group of totally positive elements in $F^\times$. Thus 
\beq\label{H0}
H^0(S) = C(F^\times_{\A_f}/F^\times_+,\C) = \varinjlim\limits_K \, C(F^\times_{\A_f}/F^\times_+ \nu(K),\C)
\eeq
is the space of continuous complex valued functions on $F^\times_{\A_f}/F^\times_+$.  In particular, there is a distinguished element $\triv\in H^0(S)$ 
given by the constant function $1$, which gives the identity element of the ring $H^\bullet(S)$. We also  write $\triv$ for the analogous class in $\ch^0(S)$. 
\begin{rem}\label{rem-chi-classes}   A little more generally, for any character $\chibold$ of $F^\times_{\A}/F^\times F^\times_{\infty,+} \nu(K)$, we obtain, via (\ref{H0}),  a class $\chibold_K \in H^0(S_K)$ and 
a class $\chibold\in H^0(S)$. Of course there are analogous classes, which we denote by the same symbol, in $\ch^0(S_K)\tt_\Q E(\chibold)$ and $\ch^0(S)\tt_\Q E(\chibold)$, where $E(\chibold)$ is the field generated by the values
of $\chibold$.  
\end{rem}

The weighted special cycles are defined in Sections 5 and 10 of \cite{kudla.rem-gen}.  For $1\le n\le m$, a Schwartz function $\ph\in S(V(\A_f)^n)^K$ and $T\in \Sym_n(F)$, 
there are classes $Z(T,\ph,K) \in \ch^{nd_+}(S_K)$ and their images $[Z(T,\ph,K)] \in H^{2nd_+}(S_K)$ under the cycle class map\footnote{ 
Note the slight shift in notation from that of \cite{kudla.rem-gen} where the image was denoted by $cl([Z(T,\ph,K)] )$.}
$$\text{cl}_k: \ch^k(S_K) \lra H^{2k}(S_K).$$
For example, for $T=0$, 
$$Z(0,\ph,K) = \ph(0)\,\cbold_{S_K}^n,$$
where $\cbold_{S_K}$ is the `co-tautological' Chern class, \cite{kudla.rem-gen} (2.3). 
The weighted cycles are compatible with pullback. For a subgroup $K'\subset K$ and the resulting projection $\pr: S_{K'} \ra S_K$, we have 
$\pr^*(Z(T,\ph,K) )= Z(T,\ph,K')$. 
Thus, there are classes in the direct limits
$$
Z(T,\ph) \in \ch^{nd_+}(S)\qquad\text{and}\qquad [Z(T,\ph)] \in H^{2n d_+}(S).
$$
 
One of the main results of \cite{kudla.rem-gen}, Proposition~5.2,  is the following product formula. For $T_i\in \Sym_{n_i}(F)$ and 
$\ph_i\in S(V(\A_f)^{n_i})$, 
\beq\label{prod-ch}
Z(T_1,\ph_1)\cdot Z(T_2,\ph_2) = \sum_{\substack{ T\in \Sym_{n_1+n_2}(F)_{\ge 0}\\ \snass T = \bpm \scr T_1&*\\{}^t*&\scr T_2\epm}} Z(T,\ph_1\tt\ph_2).
\eeq
There is a corresponding formula for the cup product of the cohomology classes $[Z(T_1,\ph_1)]$ and $[Z(T_2,\ph_2)]$. 
Thus, the span of the special cycle classes, together with the class $\triv$ for $n=0$, define subrings
which we call {\it special cycle class rings}. 

\begin{rem}\label{twisted-SCs} Using the classes $\chibold_K \in H^0(S_K)$ and $\chibold \in H^0(S)$, we can define slight variants of the weighted cycles
\begin{align*}
Z(T,\ph,K,\chibold) &= Z(T,\ph,K)\cdot \chibold_K\quad\, \in \ch^0(S_K)\tt_\Q E(\chibold),\\ 
\nass 
Z(T,\ph, \chibold) &= Z(T,\ph)\cdot \chibold \qquad\quad \ \in \ch^0(S)\tt_\Q E(\chibold),
\end{align*}
and their analogues in cohomology, where we are shifting by a character of the component group.  
By taking suitable linear combinations of these cycles we can obtain special cycles supported on a given connected component. 
\end{rem}

There a (formal) generating series for special cycle classes. For $\tau= (\tau_\s)_{\s\in \Sigma} \in \H_n^d$, where $\H_n$ is the Siegel space of genus $n$, 
and $T\in \Sym_n(F)$, let 
$$\qq^T = e(\,\sum_\s \tr(\s(T)\tau_\s)\,), \qquad e(t) = e^{2\pi i t}.$$
Then, for $\ph\in S(V(\A_f)^n)$,  define
\beq\label{CH-gen-fun}
\phi_n(\tau;\ph) = \sum_{T\in \Sym_n(F)_{\ge 0}} Z(T,\ph)\cdot \qq^T \ \in \ \ch^{nd_+}(S)[[\qq]],
\eeq
a formal power series with coefficients in the Chow group, and the corresponding series 
\beq\label{B-gen-fun}
\phi^B_n(\tau;\ph) = \sum_{T\in \Sym_n(F)_{\ge 0}} [Z(T,\ph)]\cdot \qq^T \ \in \ H^{2nd_+}(S)[[\qq]],
\eeq
with coefficients in the Betti cohomology groups. 
The series $\phi_n^B(\tau;\ph)$ is the $\qq$-expansion of a 
Hilbert-Siegel modular form of parallel weight $(\frac{m}2+1, \dots, \frac{m}2+1)$. 
As shown in \cite{kudla.rem-gen}, the modularity of the Chow group valued series $\phi_n(\tau;\ph)$ follows from the Bloch-Beilinson conjecture about injectivity of the 
Abel-Jacobi map.  

The Betti version of the product formula (\ref{prod-ch}) is equivalent to the identity
$$\phi^B_n(\bpm \tau_1&{}\\{}&\tau_2\epm, \ph_1\tt\ph_2) = \phi^B_{n_1}(\tau_1;\ph_1)\cdot \phi^B_{n_2}(\tau_2,\ph_2)$$
on Hilbert-Siegel modular forms, where the product on the right side is given by the cup product in $H^\bullet(S)$.  

\subsection{Results}\label{subsec-1.2}
It is notable that the weight of the generating series $\phi_n^B(\tau;\ph)$ depends only on $m$ and not on $d_+$. 
Similarly, $d_+$ plays almost no role in the structure of the product formula (\ref{prod-ch}) or its analogue in cohomology.  
This suggests that there may be further relations among the subrings of cohomology and of Chow groups associated to 
the Shimura varieties with differing archimedean data.  
We sometimes write $\text{Sh}(V)$ or $\text{Sh}(V)_K$ rather than 
$S$ or $S_K$ for the Shimura variety associated to $V$, when we want to vary $V$. 

Since our access to the structure of these rings will be via intersection numbers, we introduce `reduced' or `numerical' versions. 
Since our results concern only the  ring in cohomology, we will now restrict our discussion to that case and,  
for convenience, we write 
$$\SC^\bullet(V)^\nat\subset H^{2\bullet d_+}(\text{Sh}(V))$$
for the subring of special cycle classes. 

On the top degree cohomology group $H^{2md_+}(S_K)$ there is a normalized degree map
\beq\label{def-deg-K}
\deg_K: H^{2md_+}(S_K) \lra \C, \qquad z= [\b] \mapsto \vol(K/K\cap Z(\Q)) \,\int_{S_K} \b,
\eeq
where $\b$ is a degree $2md_+$-form on $S_K$ representing the class $z$. Here $Z$ is the identity component of the center of $G$
and $\vol(K)$ is computed with respect to 
a Haar measure on $G(\A_f)$ that will be specified in Section~\ref{subsec2.1} below. With this normalization, the degree is invariant under pullback for the 
coverings $\pr:S_{K'}\ra S_K$, for $K'\subset K$ and hence gives a well defined map 
\beq\label{def-deg}
\deg: H^{2md_+}(S) \lra \C.
\eeq
We extend this map by zero on $H^k(S)$ for $k<2md_+$. 

Define an inner product on the cohomology ring $H^{\bullet}(S)$ by 
$$\gs{Z_1}{Z_2} := \deg(Z_1\cdot Z_2),$$
and consider its restriction to subring of special cycles. 
By associativity of the cup product, the radical of this pairing on $\SC^\bullet(V)^\nat$ is an ideal.  We 
define the {\it reduced ring of special cycles} 
\beq\label{red-ring}
\SC^\bullet(V) = \SC^\bullet(V)^\nat/\text{Rad}
\eeq  
as the quotient of $\SC^\bullet(V)^\nat$ by this ideal.
The form $\gs{}{}$ then defines a 
non-degenerate pairing on $\SC^\bullet(V)$. 
We let 
$$z(T;\ph) = \text{ the image of $[Z(T;\ph)]$ in $\SC^\bullet(V)$,}$$ 
and write $z(T;\ph) =  z_V(T;\ph) $ if we want to keep track of the space $V$. 
Note that, by definition, $\SC^0(V)^\nat = \C\, \triv$, for the class $\triv$ defined above. Thus 
$\ker(\deg:\SC^m(V)^\nat \ra \C)$ is the intersection of $\SC^m(V)^\nat$ with the radical
and hence $\SC^m(V) = \C\,\triv^\vee$ for a class $\triv^\vee$ with $\gs{\triv}{\triv^\vee}=1$. 
For example, if $T\in \Sym_m(F)_{>0}$ and $\ph\in S(V(\A_f)^n)^K$, then $Z(T,\ph)$ is a weighted $0$-cycle on $S_K$, and 
$$z(T,\ph) = \deg(Z(T,\ph))\cdot \triv^\vee.$$

We want to consider how the ring $\SC^\bullet(V)$ varies with $V$.  

Our main result concerns the case where only the archimedean part of $V$ varies. 
Note that the real quadratic spaces of signatures $(m+2,0)$ and $(m,2)$ 
both have determinant $1$ and have Hasse invariants $+1$ and $-1$ respectively. Also recall that the character $\chi_V$ of a quadratic 
space over $F$ is define by 
$$\chi_V(x) = (x, (-1)^s\det(V))_F,\qquad x\in F^\times_\A,$$
where $s= \frac12\dim V (\dim V+1)$ and $(\,,\,)_F$ is the global quadratic Hilbert symbol. For quadratic spaces $V$ and $V'$ of dimension $m+2$ over $F$ with 
$\chi_V = \chi_{V'}$, 
we say that $V\cong_f V'$
if there is an isometry $V'_{\pp} \simeq V_\pp$ for each finite place $\pp$ of $F$. Here $V_\pp = V\tt_F F_\pp$.  
For spaces $V$ and $V'$ with $V\cong_f V'$, we fix an isomorphism $V(\A_f) \simeq V'(\A_f)$ and hence obtain isomorphisms 
\beq\label{naive-matching}
\rho^n_{V,V'}:S(V(\A_f)^n) \simeq S(V'(\A_f)^n), \qquad \ph \mapsto \ph',
\eeq
for all $n$, compatible with tensor products and with the action of the metaplectic group via the Weil 
representation. Note that, by the product formula for the Hasse invariant,  $V\cong_f V'$ implies that $d_+(V)$ and $d_+(V')$ have the same parity
but are otherwise unconstrained. The following result will be proved in Section 3 as a consequence of Theorem~A of Section 2 and the matching principle. 
 
\begin{theob}\label{thmA} Suppose that $V\cong_f V'$ and that 
$1\le d_+(V), d_+(V') <d$. \hfb
Then 
the map 
$$\SC^\bullet(V) \lra \SC^\bullet(V'), \qquad z_V(T,\ph) \mapsto z_{V'}(T,\ph'),\qquad \ph'= \rho^n_{V,V'}(\ph),$$
is well defined.   Moreover, this map is an isometry and a ring isomorphism. 
\end{theob}
A striking aspect of these isomorphisms is that, when $d_+(V)\ne d_+(V')$,  they relate classes in different degrees!

{\bf Example.} As a concrete example, consider the case where $F$ is the maximal real subfield of $\Q(\mu_{19})$, so that $d=9$, and take $m=3$. 
Let $V$ be a quadratic space over $F$ with signature $((3,2),(5,0)^{8})$ so that $d_+(V)=1$. 
Then, for a fixed neat compact open subgroup $K$ in 
$G(\A_f)$, the variety $Sh(V)_K$ is a smooth projective $3$-fold, perhaps with several connected components. It can be viewed as bearing the same relation to 
the classical Siegel $3$-fold over $\Q$ as a Shimura curve over $F$ bears to the classical modular curve. The special cycles are generalized Humbert surfaces, 
for $n=1$, Shimura curves, for $n=2$, and $0$-cycles, for $n=3$. Their cohomology classes, taken up to numerical equivalence,  
$$\SC^n(V)^K = \text{span}\{\,Z(T,\ph)\, \mid T\in \Sym_n(F)^K_{\ge0}, \ \ph\in S(V(\A_f)^n)\,\})/\sim,$$ 
yield a reduced intersection ring, $\SC^\bullet(V)^K$, and, passing to the limit over $K$, 
$$\SC^\bullet(V) = \SC^0(V) \oplus \SC^1(V)\oplus \SC^2(V)\oplus \SC^3(V), \qquad \SC^0(V) = \C \,\triv, \ \SC^3(V)= \C \,\triv^\vee.$$ 
Note that $\SC^\bullet(V)$is a {\it subquotient} of the cohomology ring of $Sh(V)$. 

For an {\it odd} integer $r$ with $1\le r <9$,  we consider quadratic spaces $V'= V'[r]$ over $F$ with a fixed isomorphism 
\beq\label{f-id}
V'(\A_f)\simeq V(\A_f)\eeq
and with 
$$\sig(V') = ((3,2)^{r},(5,0)^{d-r}).$$ 
Varying $r$ and the subset $\Sigma_+(V'[r])$ of indefinite places, we have $\sum_r{9 \choose r} =255$ such spaces, up  to isomorphism. 
Identifying $G(V)(\A_f)$ and $G(V')(\A_f)$ via (\ref{f-id}), we have a collection of Shimura varieties $Sh(V')_K$ 
of dimensions $3 r = 3, 9, 15,$ and $21$ with isomorphic reduced intersection rings, $\SC^\bullet(V')^K$. 

\subsection{Outline of contents}
This ends our extended introduction. In Section~\ref{section2}, we review the Siegel-Weil formula and show that it identifies the image  
of the generating series $\phi^B_n(\tau;\ph)$ of (\ref{B-gen-fun}) under the degree map with a special value of a Siegel Eisenstein series, 
Proposition~\ref{prop2.2}. By considering Fourier coefficients of pullbacks, we obtain formulas for inner products and triple products
of special cycles classes in term of Fourier coefficients of such pullbacks, Theorem~A and Corollary~\ref{IP-cor}. In Section~\ref{section3}, 
we explain how the matching principle introduced in \cite{K.Bints} yields comparisons like that of Theorem~B.   
In Section~\ref{section4}, we consider the case $d_+=0$, so that $V=\Vp$ is totally positive definite. For such a space 
we give a combinatorial description of a `cohomology' ring and `special cycles' classes in it.  Again by the Siegel-Weil formula, 
we show that Theorem~B can be extended to the case $d_+=0$. In particular, the special cycle ring $\SC^\bullet(V)$ for $d_+(V)$ even 
is isomorphic to the `special cycle' ring $\SC^\bullet(\Vp)$ for the associated totally positive definite space. 
Finally, in Section~\ref{section5}, we explain how information about local matching 
can be obtained from the results of \cite{KR.rdps} and \cite{sweet.meta} on degenerate principal series representations. 

\subsection{Thanks} I would like to thank Siddarth Sankaran for useful discussion about the construction of Section~\ref{section4}. 

\section{The Siegel-Weil formula and intersection numbers}\label{section2}
In this section, we explain how the Siegel-Weil formula provides information about the intersection products of special cycles. 
The case $d_+=1$ is treated in Section~10 of \cite{K.duke} to which we refer to more details.  We will use some of the notation of Section~5.3 of \cite{kudla.rem-gen}.

\subsection{Measures}\label{subsec2.1}
Let $A^{2md_+}(S_K)$ be the space of smooth top degree differential forms on $S_K$. 
By \cite{kudla.rem-gen} (5.18), the form $\Om^n$, defined there, is a $(nd_+,nd_+)$-form on $S_K$ which represents 
the $n$-th power of the top Chern class of the vector bundle 
$\boldCC_S$. In particular, the form $(-1)^{md_+} \Om^m$ gives an invariant volume form on $D$
and descends to $S_K$. 
For a fixed base point $z_0\in D^+$, let $K_\infty$  be the stabilizer of 
$z_0$ in $G(\R)$. Note that $K_\infty$ contains $Z(\R)$. 
Then there is an isomorphism
\beq\label{J-iso}
J:A^{2md_+}(S_K) \isoarrow  [C^\infty(G(\Q)Z(\R)\back G(\A))]^{K_\infty K},
\eeq
defined as follows. 
If $\eta$ is a $2md_+$-form on $S_K$, 
write 
$\eta = \breve\eta\cdot (-1)^{md_+}\Om^{md_+}$ for a function $\breve\eta$ on $S_K$
and define $J(\eta)$ as the pullback of $\breve\eta$ to $G(\Q)Z(\R)\back G(\A)$ via the natural surjective map
$$G(\Q)Z(\R)\back G(\A) \lra S_K = G(\Q)\back D\times G(\A_f)/K, \qquad [g_\infty, g_f] \mapsto [g_\infty(z_0),g_f].$$

We define a Haar measure\footnote{Here, to lighten notation, 
 we slightly abuse notation and write $\SO(V)$ rather than $R_F/\Q \SO(V)$.} $d_\infty g$ on $\SO(V)(\R) \simeq Z(\R)\back G(\R)$ as follows. 
For a smooth compactly supported form $\a \in A^{2md_+}_c(D)$ on $D$, write 
$\a = \breve \a \cdot (-1)^{md_+} \Om^{md_+}$ for a compactly supported function $\breve\a$ on $Z(\R)\back G(\R)$. 
Then $d_\infty g$ is defined by the condition that 
\beq\label{int-D-identity}
\int_D \a = \int_{Z(\R)\back G(\R)} \breve \a(g)\,d_\infty g,
\eeq
for all such $\a$. 

Let $d^Tg$ be Tamagawa measure on $\SO(V)(\A)$. The factorization $d^Tg = d_\infty g\cdot d_f g$, for $d_\infty g$ just defined,
determines a unique Haar measure on $\SO(V)(\A_f) = Z(\A_f)\back G(\A_f)$. 
Now 
\beq\label{haar-centerl}
Z(\Q)\back Z(\A_f) =Z(\Q) Z(\R)\back Z(\A) \simeq F^\times_\A/F^\times F^\times_\R
\eeq
is compact, and, taking the Haar measure $dz$ giving this space volume $1$, we obtain a Haar measure $d_f \tilde g$
on $Z(\Q)\back G(\A_f)$ with $d_f\tilde g = d_f g\cdot dz$. 
Finally, we have the measure $dg = d_\infty g_\infty \cdot d_f\tilde g_f$  on $Z(\R)Z(\Q)\back G(\A)$. 
With these choices, we have, for $\eta\in A^{2md_+}(S_K)$, 
\beq\label{geo-int}
\vol(K/(K\cap Z(\Q)), d_f\tilde g)\cdot \int_{S_K}\eta =  \int_{G(\Q)Z(\R)\back G(\A)} J(\eta)(g)\,dg.
\eeq

\subsection{Degree formulas}\label{section2}
We now follow Section~5.3 of \cite{kudla.rem-gen}.  Recall that for $G'=G'_n= R_{F/\Q} \Sp(n)$, the global metaplectic group 
$\wtg(\A)$ acts on $S(V(\A)^n)$ via the Weil representation $\o=\o_V=\o_{\psi, V}$ for a fixed nontrivial character $\psi$ of $F_\A/F$. 
Also recall the Schwartz form
\beq\label{arch-SF}
\phbold^{(n)}_\infty = \bigotimes_{\s\in \Sigma_+(V)}\ph_\s^{(n)} \ \tt \bigotimes_{\s\notin \Sigma_+(V)} \ph_{\s,+}^0 
\in [S(V_\R^n) \tt A^{(nd_+,nd_+)}(D)]^{G(\R)},
\eeq
where, for $\s\in \Sigma_+(V)$, 
$$\ph_\s^{(n)}\ \in S(V_\s^n)\tt A^{(n,n)}(D_\s)$$
 is the Schwartz form for $V_\s$, and, for $\s\notin \Sigma_+(V)$, $\ph_{\s,+}^0\in S(V_\s^n)$ is the Gaussian for $V_\s$, cf. Sections 7 and 8 of \cite{K.duke}.
 For $\ph \in S(V(\A_f)^n)^K$, the theta form  
\beq\label{theta-form}
\theta(g';\ph) = \sum_{x\in V(F)^n} \o(g')(\phbold^{(n)}_\infty\tt\ph)(x), \qquad g'\in \wtg(\A), 
\eeq
defines a closed $(nd_+,nd_+)$-form on $S_K$ and, as a function of $g'$,  is left invariant for the (canonical) image of $G'(\Q)$ in $\wtg(\A)$. 

We apply (\ref{J-iso}) to the top degree form $\theta(g';\ph) \wedge \Om_S^{n-m}$ on $S_K$ and obtain, for $g\in G(\A)$,  
\beq\label{to-Schwartz}
J(\, \theta(g';\ph) \wedge \Om_S^{n-m})(g) =(-1)^{md_+}\, \theta(g',g;\breve\phbold^{(n)}_\infty\tt\ph),
\eeq
where
$\breve\phbold^{(n)}_\infty$ is 
defined by 
\beq\label{arch-SF-fun}
\breve\phbold^{(n)}_\infty = \bigotimes_{\s\in \Sigma_+(V)}\breve\ph_\s^{(n)} \ \tt \bigotimes_{\s\notin \Sigma_+(V)} \ph_{\s,+}^0 
\in S(V_\R^n),
\eeq
with
\beq\label{phKMtoS}
\ph^{(n)}_\s(x)\wedge \Om^{m-n} = \breve\ph^{(n)}_\s(x)\,\Om^m.
\eeq
In particular, 
$\breve\phbold^{(n)}_\infty\tt\ph \in S(V(\A)^n)$, 
and the function in (\ref{to-Schwartz}) is the usual theta function
\beq\label{J-theta-form}
\theta(g',g;\breve\phbold^{(n)}_\infty\tt\ph) = \sum_{x\in V(F)^n} \o(g')(\breve\phbold^{(n)}_\infty\tt\ph)(g^{-1}x). 
\eeq
Also note that 
$$\phbold^{(n)}_\infty(0)\wedge \Om^{m-n} = \Om^m.$$
This accounts for the sign change, due to the fact that, in the definition of $J$,  $(-1)^{md_+}\Om^m$, the top power of the K\"ahler form is used.

Now, by (5.20) of \cite{kudla.rem-gen}, the cohomology class of the theta form is the generating series
\beq\label{KM-theo}
[\theta(g'_\tau;\ph)] = N(\det(v))^{\frac{m+2}4} \,\phi^B_n(\tau;\ph) = N(\det(v))^{\frac{m+2}4} \,
\sum_{T\in \Sym_n(F)_{\ge 0}} [Z(T,\ph)]\, \qq^T.
\eeq
Here, as in \cite{K.Bints}, section 1 in the case $n=1$,  $g'_\tau=[g_\tau,1]$ is an element of the metaplectic cover $\widetilde{G'}(\R)$ 
where $g_\tau\in G'(\R)$ has components
$$(g_\tau)_\s = \bpm 1&u_\s\\ {}&1\epm \bpm a_\s&{}\\{}&{}^ta_\s^{-1}\epm, \qquad \tau_\s = u_\s + i v_\s\, \in \H_n, \quad v_\s = a_\s\,{}^ta_\s.$$ 
Thus, recalling (\ref{def-deg-K}) and using (\ref{geo-int}),  we have 
\beq\label{deg-form-1}
\deg(\phi_n^B(\tau,\ph)\cdot \cbold_{S}^{m-n}) = (-1)^{md_+}N(\det(v))^{-\frac{m+2}4}\, \int_{G(\Q)Z(\R)\back G(\A)} \theta(g_\tau',g;\breve\phbold^{(n)}_\infty\tt\ph)\,dg.
\eeq

Similarly, we note that the Schwartz forms are compatible with wedge products so that, 
for $n=n_1+n_2$, 
$$\phbold^{(n)}_\infty = \phbold^{(n_1)}_\infty \wedge \phbold^{(n_2)}_\infty.$$
Thus, $\tau_1\in \H_{n_1}^d$, $\tau_2\in \H_{n_2}^d$,  and $\ph_i\in S(V(\A_f)^{n_i})$, $i=1$, $2$, 
\beq\label{theta-prod}
\theta(g'_{\tau_1};\ph_1) \wedge \theta(g'_{\tau_2};\ph_2) = \theta(g'_\tau; \ph_1\tt\ph_2),\qquad \tau = \bpm \tau_1&{}\\{}&\tau_2\epm.
\eeq
Again computing the degree, we have, for $n_1+n_2=m$, 
\begin{align}\label{deg-form-2}
&\gs{\,\phi^B_{n_1}(\tau_1,\ph_1)}{\phi^B_{n_2}(\tau_2,\ph_2)}\\
\nass
{}&=
(-1)^{md_+}N(\det(v_1)\det(v_2))^{-\frac{m+2}4}\, \int_{G(\Q)Z(\R)\back G(\A)} \theta(g_\tau',g;\breve\phbold^{(m)}_\infty\tt\ph_1\tt\ph_2)\,dg,\notag
\end{align}
for $\tau$ as in (\ref{theta-prod}). 

\subsection{Consequences of the Siegel-Weil formula}

We next apply the Siegel-Weil formula as in section 10 of \cite{K.duke}. 
Let  $I_n(s,\chi_V)$ be the degenerate principal series representation of $\wtg(\A)$ associated to $\chi_V$, induced from the Siegel 
parabolic $P$ of $G'$.  
Define the $\wtg(\A)$ intertwining map
$$\l_V: S(V(\A)^n)\lra I_n(s_0,\chi), \qquad \phbold\mapsto \l_V(\phbold)(g') = \o(g')\phbold(0),\qquad s_0 = \frac12\dim V - \rho_n,$$
$\rho_n= \frac12(n+1)$. The standard section $\P(s;f) \in I_n(s,\chi)$ attached to $\phbold$ is given  by 
$$\P(g',s;\phbold) = \o(g')\phbold(0)\cdot |a(g')|^{s-s_0}.$$
Here we follow the notation of \cite{KR.crelle}, \cite{sweet.thesis}, and \cite{gan.qiu.takeda}.
The Siegel-Eisenstein series is defined by the series
$$E(g',s,\l_V(\phbold)) = \sum_{\gamma\in P'(F)\back G'(F)} \P(\gamma g',s;\phbold),$$
in the half-plane of absolutely convergence $\Re(s)>\rho_n$ and has a meromorphic analytic continuation to 
the whole $s$ plane. 
Since $V$ is anisotropic, the Eisenstein series is holomorphic at $s=s_0$ and, by the Siegel-Weil formula, \cite{KR.crelle}, 
$$E(g',s_0,\l_V(\phbold)) = \int_{O(V)(F)\back O(V)(\A_F)} \theta(g',g;\phbold)\,dg$$
where 
$$\vol(O(V)(F)\back O(V)(\A_F),dg)=1.$$
In fact, we have
\beq\label{Tamagawa-2}
\int_{O(V)(F)\back O(V)(\A_F)} \theta(g',g;\phbold)\,dg = \frac12\int_{\SO(V)(F)\back SO(V)(\A_F)} \theta(g',g;\phbold)\,d^Tg.
\eeq
where $d^Tg$ is the Tamagawa measure on $\SO(V)(\A_F)$. 
This follows from the argument of Section~4 of \cite{K.Bints}, where we use the fact that the sign representation of $O(V)(F_\pp)$ 
does not occur in the local theta correspondence with $G'_\pp$ for $n<\dim V$, \cite{rallisHDC}, \cite{rallisHDC.bis}. 

On the other hand, since the action of $G(\A)$ on $S(V(\A)^n)$ factors through $\SO(V)(\A)$, the integral on the right side of 
(\ref{deg-form-1}), is 
\beq\label{basic-theta-int}
\int_{G(\Q)Z(\R)\back G(\A)} \theta(g',g;\breve\phbold_\infty\tt\ph)\,dg = 
\int_{\SO(V)(F)\back \SO(V)(\A)} \theta(g',g;\breve\phbold_\infty\tt\ph)\,d^Tg.
\eeq

Since the archimedean component $\l_{V_\infty}(\breve\phbold_\infty)$ will be fixed, for $\ph \in S(V(\A_f)^n)$, 
for 
$\kappa = \frac{m}2+1$ and  $s_0=\kappa-\rho_n$, 
we let
\beq\label{good-Eis}
E(\tau,s_0,\l_{V_f}(\ph)):= (-1)^{md_+} N(\det(v))^{-\kappa/2}\cdot
E(g'_\tau,s_0,\l_V(\breve\phbold_\infty\tt\ph)).
\eeq
\begin{rem}
Note that, for convenience,  we have included the sign $(-1)^{md_+}$ in the definition. This sign must be kept in 
mind later. 
\end{rem}
Then, from (\ref{deg-form-1}),  we obtain the following expression for the degree in terms of the Eisenstein series of genus $n$. 
\begin{prop}\label{prop2.2}
  For $\tau\in \H_n^d$, 
\beq\label{deg-form-2}
\deg(\phi_n^B(\tau,\ph)\cdot \cbold_{S}^{m-n}) = 2\, E(\tau,s_0,\l_{V_f}(\ph)),
\eeq
where $s_0= \kappa-\rho_n$.
\end{prop}
In particular, as a consequence of this identity, the special value $E(\tau,s_0,\l_{V_f}(\ph))$ is a {\it holomorphic} Hilbert-Siegel modular form of parallel weight $\kappa$ 
with Fourier 
expansion
\beq\label{Eis-FC}
E(\tau,s_0,\l_{V_f}(\ph)) =\sum_{T\in \Sym_n(F)_{\ge0}} A(T,\l_{V_f}(\ph))\, \qq^T.
\eeq
\begin{cor}\label{cor2.2}  The inner product of a class $[Z(T,\ph)]\in H^{2nd_+}(S)$ with the class $\cbold_{S}^{m-n}$ is given by 
$$\gs{Z(T,\ph)}{\cbold_{S}^{m-n}} = \deg(Z(T,\ph)\cdot \cbold_S^{m-n}) =2\, A(T, \l_{V_f}(\ph)).$$
\end{cor}

We also obtain relations between intersection products and Fourier coefficients of pullbacks. 
For $\ph_i\in S(V(\A_f)^{n_i})$, $i=1$, $2$, where $n_1+n_2=m$, we write
\beq\label{Eis-FC-2}
E(\bpm \tau_1&{}\\ {}&\tau_2\epm,\frac12,\l_{V_f}(\ph_1\tt\ph_2)) = \sum_{T_1\in \Sym_{n_1}(F)_{\ge0}} \sum_{T_2\in \Sym_{n_2}(F)_{\ge0}}
A(T_1,T_2;\l_{V_f}(\ph_1\tt\ph_2))\,\qq_1^{T_1}\,\qq_2^{T_2}
\eeq
for the Fourier expansion of the pullback under (\ref{diag-2}). 
Similarly, for  $\ph_i\in S(V(\A_f)^{n_i})$, $i=1$, $2$, $3$, where $n_1+n_2+n_3=m$, 
we write 
\begin{multline}\label{Eis-FC-3}
E(\bpm \tau_1&{}&{}\\ {}&\tau_2&{}\\{}&{}&\tau_3\epm,\frac12,\l_{V_f}(\ph_1\tt\ph_2\tt\ph_3))\\
\nass
 = \sum_{T_1\in \Sym_{n_1}(F)_{\ge0}} \sum_{T_2\in \Sym_{n_2}(F)_{\ge0}} \sum_{T_3\in \Sym_{n_3}(F)_{\ge0}}
A(T_1,T_2, T_3;\l_{V_f}(\ph_1\tt\ph_2\tt\ph_3))\,\qq_1^{T_1}\,\qq_2^{T_2}\,\qq_3^{T_3}
\end{multline}
for the Fourier expansion of the pullback under (\ref{diag-3}).
Note that the special value is now take at the point $s_0=\kappa - \rho_m = \frac12$. Also note that these pullbacks have a non-trivial cuspidal 
component.  For example, in case (\ref{Eis-FC}), the cuspidal projection involves the doubling integrals \cite{PSR-doubling}, \cite{KR.annals}, \cite{gan.qiu.takeda}, and hence is 
controlled by special values of L-functions.  

\begin{atheo}\label{theo-B} (i)  For $\ph_i\in S(V(\A_f)^{n_i})$, $i=1$, $2$, where $n_1+n_2=m$, 
\begin{align*}
\deg(\,\phi_{n_1}^B(\tau_1,\ph_1) \cdot \phi_{n_2}^B(\tau_2,\ph_2)\,) &= 2\,  E(\bpm \tau_1&{}\\ {}&\tau_2\epm,\frac12,\l_{V_f}(\ph_1\tt\ph_2)).
\end{align*}

In particular, for $T_1\in \Sym_{n_1}(F)_{\ge 0}$ and $T_2\in \Sym_{n_2}(F)_{\ge 0}$, the intersection product of the 
weighted special cycles is given by
$$\deg(\,[Z(T_1,\ph_1)]\cdot [Z(T_2,\ph_2)]\,) =2\, A(T_1,T_2;\l_{V_f}(\ph_1\tt\ph_2)),$$
for the Fourier coefficient of the pulback in (\ref{Eis-FC-2}). \hfb 
(ii) Similarly, for  $\ph_i\in S(V(\A_f)^{n_i})$, $i=1$, $2$, $3$, where $n_1+n_2+n_3=m$, 
\begin{align*}
\deg(\,\phi_{n_1}^B(\tau_1,\ph_1) \cdot \phi_{n_2}^B(\tau_2,\ph_2)\cdot \phi_{n_3}^B(\tau_3,\ph_3)\,) &= 2\,  
E(\bpm \tau_1&{}&{}\\ {}&\tau_2&{}\\ {}&{}&\tau_3\epm,\frac12,\l_{V_f}(\ph_1\tt\ph_2\tt\ph_3)),
\end{align*}
and hence
$$\deg(\,[Z(T_1,\ph_1)]\cdot [Z(T_2,\ph_2)]\cdot [Z(T_3,\ph_3)]\, ) = 2\,A(T_1,T_2, T_3;\l_{V_f}(\ph_1\tt\ph_2\tt \ph_3)),$$
for the Fourier coefficient of the pullback in (\ref{Eis-FC-3}).
\end{atheo}

Passing to classes in $SC^\bullet(V)$, we obtain the following formulas. 
\begin{cor}\label{IP-cor} (i) For $\ph_i\in S(V(\A_f)^{n_i})$, $i=1$, $2$, where $n_1+n_2=m$,
$$ 
 \gs{z(T_1,\ph_1)}{z(T_2,\ph_2)} = 2\,A(T_1,T_2;\l_{V_f}(\ph_1\tt\ph_2)).
$$
 (ii) For  $\ph_i\in S(V(\A_f)^{n_i})$, $i=1$, $2$, $3$, where $n_1+n_2+n_3=m$,
$$
 \gs{z(T_1,\ph_1)\cdot z(T_2,\ph_2)}{z(T_3,\ph_3)} = 2\,A(T_1,T_2, T_3;\l_{V_f}(\ph_1\tt\ph_2\tt \ph_3)).
 $$ 
\end{cor} 

Part (ii) of Corollary~\ref{IP-cor}  
shows that the products in the ring $SC^\bullet(V)$ are completely controlled by Fourier coefficients of pullbacks under 
(\ref{diag-3}) of 
certain holomorphic Hilbert-Siegel Eisenstein series of genus $m$ and of 
parallel weight $(\kappa,\dots, \kappa)$, $\kappa = \frac{m}2+1 = \rho_m+\frac12$. 
More precisely, note that a class $z\in SC^n(V)$ is determined by its inner products $\gs{z}{z'}$ for $z'\in SC^{m-n}(V)$. 
In particular, the class $z(T_1,\ph_1)\cdot z(T_2,\ph_2) \in SC^{n_1+n_2}(V)$ is determined by the inner products 
with classes of the form $z(T_3,\ph_3)$ and these are 
given by (ii) of the corollary. 
Similarly, the inner product on $SC^\bullet(V)$ is determined by (i) of the corollary. 

\begin{rem}
It will be interesting to investigate the non-vanishing of such Fourier coefficients of pullbacks to see what kind of information 
about the reduced ring of special cycles can be obtained from our product formulas. 
\end{rem}

\section{Comparison results}\label{section3}

Since the structure of the reduced ring of special cycles is controlled by the Siegel-Weil Eisenstein series,  
the matching principle introduced in \cite{K.Bints}, section 4, yields relations 
between such special cycle rings. We will use the notation of section 4 of \cite{K.Bints} generalized to $F$. A detailed 
treatment can be found in  \cite{KR.crelle}. 

\subsection{The matching principle} 
Suppose that $V_1$ and $V_2$ are quadratic spaces over $F$ with $\dim_FV_1=\dim_FV_2$ and $\chi=\chi_{V_1}=\chi_{V_2}$. 
There are intertwining maps
$$\l_{V_i}:  S(V_i(\A)^n) \lra I(s_0,\chi), \qquad \l_{V_i}(\ph_i)(g') = \o_{V_i}(g')\ph_i(0),$$
where $s_0= \kappa - \rho_n$. 
\begin{defn} (\cite{K.Bints}, Definition~4.3.)    Schwartz functions $\ph_1\in S(V_1(\A)^n)$ and $\ph_2\in S(V_2(\A)^n)$ 
are said to {\bf match} if 
$$\l_{V_1}(\ph_1) = \l_{V_2}(\ph_2) \in I(s_0,\chi).$$
\end{defn}
There are obvious local analogues. 

For example, in the archimedean case, for $V=V_i$, $i=1$, $2$, suppose that the signature of $V$ is given by (\ref{def-sigV}). 
For $\s\in \Sigma_+(V)$, let  $D^{(\s)} = D(V_\s)$ be the space of 
oriented negative $2$-planes in $V_\s= V\tt_{F,\s}\R$ as in (\ref{DVS}).  
For $\s\in \Sigma_+(V)$, let $\breve\ph^{(n)}_\s$ be the Schwartz function on $V_\s^n$ defined by (\ref{phKMtoS}).
Let $V_{\s,+}$ be the quadratic space over $\R$ of signature $(m+2,0)$, and let 
$$\ph^{(n)}_{\s,+} \in S(V_{\s,+}^n)$$
be the Gaussian. Then, by the analogue of Corollary~4.15 of \cite{K.Bints}, we have the following matching.
\begin{lem}
$$\l_{V_\s}(\breve\ph^{(n)}_\s) = \l_{V_{\s,+}}(\ph^{(n)}_{\s,+})\ \in I_{n,\s}(s_0,\chi).$$
\end{lem}
This is immediate from the fact that these two functions are eigenvectors for $K'_\infty$ of weight $\frac{m}2+1$ and 
$\breve\ph^{(n)}_\s(0) = \ph^{(n)}_{\s,+}(0) =1$. 
If we write 
$$\phbold_{\infty, i} \in S((V_i\tt_\Q \R)^n)\tt A^{(nd_+^i,n d_+^i)}(D(V_i\tt_\Q\R)), \qquad i=1, 2, $$
for the Schwartz forms defined by (\ref{arch-SF}) and $\breve\phbold_{\infty,i}\in S(V_{i,\R}^n)$ for the corresponding Schwartz functions, then these match as well. 

\begin{rem}  Note that matching is {\bf not}  compatible with tensor products in general.
If $\ph_{\pp,i}^{(n_1)}\in S(V_{\pp,i}^{n_1})$ are matching and $\ph_{\pp,i}^{(n_2)}\in S(V_{\pp,i}^{n_2})$ are matching, 
it need not be the case that the tensor products 
$\ph_{\pp,1}^{(n_1)}\tt \ph_{\pp,1}^{(n_2)} \in S(V_{\pp,1}^n)$ and $\ph_{\pp,2}^{(n_1)}\tt \ph_{\pp,2}^{(n_2)}\in S(V_{\pp,2}^n)$ 
are matching. A rather explicit description of examples is given in Section~\ref{section5}, in particular Remark~\ref{rem5.3} (iii). 
\end{rem}

{\bf Basic Observation:}  (i) If $\ph_1\in S(V_1(\A)^n)$ and $\ph_2\in S(V_2(\A)^n)$ are matching Schwartz functions, 
then the associated Siegel-Eisenstein series coincide, 
\beq\label{basic-matching-id-general}
E(g',s, \l_{V_1}( \ph_1)) = E(g',s, \l_{V_2}(\ph_2)).
\eeq
(ii) If $V_1$ and $V_2$ 
have signatures given by (\ref{def-sigV}) for $d_+(V_1)$ and $d_+(V_2)$, 
and if $\ph_1\in S(V_1(\A_f)^n)$ and $\ph_2\in S(V_2(\A_f)^n)$ are matching Schwartz functions on the finite ad\`eles, then the associated Siegel-Eisenstein series coincide, 
\beq\label{basic-matching-id}
E(g',s, \l_{V_1}(\breve\phbold_{\infty,1}\tt \ph_1)) = E(g',s, \l_{V_2}(\breve\phbold_{\infty,2}\tt\ph_2)).
\eeq

\subsection{Consequences of matching} \label{section3.2}
Since the inner product and the ring structure on $SC^\bullet(V)$ are controlled by the Fourier coefficients of 
pullbacks of special values of such Siegel-Eisenstein series, the identity (\ref{basic-matching-id}) implies various relations. 

For the remainder of this section, we slightly shift notation and suppose that $V$ and $V'$ are quadratic spaces over $F$ with $\dim_FV=\dim_FV'$,  $\chi=\chi_{V}=\chi_{V'}$, and with signatures given by (\ref{def-sig-V}) with $1\le d_+(V), d_+(V') <d$.
For data $T\in \Sym_n(F)$ and $\ph\in S(V(\A_f)^n)$, we write $Z_V(T,\ph)$ (resp.
$z_V(T,\ph)$)  for the corresponding class in $H^{2nd_+(V)}(Sh(V))$ (resp. $\SC^n(V)$).  

First, we have the following consequence of Corollary~\ref{cor2.2}.
\begin{prop}\label{prop3.5}
If $\ph\in S(V(\A_f)^n)$ and $\ph'\in S(V'(\A_f)^n)$ are matching functions,  
then,  for all $T\in \Sym_n(F)_{\ge0}$, 
$$(-1)^{m d_+(V)}\deg(Z_V(T,\ph)\cdot \cbold_S^{m-n}) = (-1)^{m d_+(V')}\deg(Z_{V'}(T,\ph')\cdot \cbold_{S'}^{m-n}).$$
\end{prop}

\begin{rem} (i)  This says that the volumes of the cycles $Z_V(T,\ph)$ and $Z_{V'}(T,\ph')$ 
with respect to a suitable power of the respective K\"ahler forms coincide.
\hfb 
(ii) For example,  note that by matching, $\ph(0) = \l_V(\ph)(1)= \l_{V'}(\ph')(1) = \ph'(0)$, whereas
$Z_V(0,\ph) = \ph(0)\,\cbold_S^n$ and $Z_{V'}(0,\ph') = \ph'(0)\,\cbold_{S'}^n$. Thus, if we can find matching $\ph$ and $\ph'$ with $\ph(0)\ne 0$,  
Proposition~\ref{prop3.5} gives
\beq\label{equal-Chern-no}
(-1)^{m d_+(V)}\deg(\cbold_S^m) = (-1)^{m d_+(V')}\deg(\cbold_{S'}^m),
\eeq
an identity between top Chern numbers. 
As explained below, for $n=1$ such matching functions always exist, so that (\ref{equal-Chern-no}) holds for any pair $V$ and $V'$. 
\end{rem} 

Next, as a consequence of Corollary~\ref{IP-cor}, we have the following relations 
between inner products in $SC^\bullet(V)$ and $SC^\bullet(V')$.  
\begin{prop}\label{prop-3.6} (i) Suppose that, for $m=n_1+n_2$,  and for pairs of functions $\ph_i\in S(V(\A_f)^{n_i})$ and 
$\ph'_i\in S(V'(\A_f)^{n_i})$, the functions $\ph_1\tt\ph_2 \in S(V(\A_f)^m)$ and $\ph'_1\tt\ph'_2 \in S(V'(\A_f)^m)$
match. Then 
$$\gs{z_V(T_1,\ph_1)}{z_V(T_2,\ph_2)}  =  \e_{V,V'}\,\gs{z_{V'}(T_1,\ph'_1)}{z_{V'}(T_2,\ph'_2)},$$
where $\e_{V,V'}=(-1)^{m (d_+(V')-d_+(V))}$. \hfb
(ii)  Suppose that, for $m=n_1+n_2+n_3$,  and for triples of functions $\ph_i\in S(V(\A_f)^{n_i})$ and 
$\ph'_i\in S(V'(\A_f)^{n_i})$, the functions $\ph_1\tt\ph_2 \tt\ph_2\in S(V(\A_f)^m)$ and $\ph'_1\tt\ph'_2\tt \ph'_3 \in S(V'(\A_f)^m)$
match. Then 
$$\gs{z_V(T_1,\ph_1)\cdot z_V(T_2,\ph_2)}{z_V(T_3,\ph_3)} = \e_{V,V'}\,\gs{z_{V'}(T_1,\ph'_1)\cdot z_{V'}(T_2,\ph'_2)}{z_{V'}(T_3,\ph'_3)} .$$
\end{prop}

Now suppose that the quadratic spaces $V$ and $V'$ are isomorphic at all finite primes so that $V\simeq_f V'$ in the 
notation of Section~\ref{subsec-1.2}. Note that in this case $d_+(V)$ and $d_+(V')$ have the same parity and 
hence $\e_{V,V'}=1$. We then have identifications (\ref{naive-matching}) as at the end of Section~\ref{subsec-1.2}
which provide an automatic matching
\beq\label{taut-matching}
\ph \ \leftrightarrow\ \ph' = \rho^n_{V,V'}(\ph).
\eeq 
for all data, 
compatible with tensor products and with the Weil representation. 
We have the following comparison result, a more precise version of Theorem~B.  
\begin{theo}\label{theo3.4}
There is a linear map
$$\rho_{V,V'}: SC^\bullet(V) \lra SC^\bullet(V')$$
such that, for $\ph$ and $\ph'$ matching as in (\ref{taut-matching}), 
$$\rho_{V,V'}:  z_V(T,\ph) \mapsto z_{V'}(T,\ph').$$
Moreover, this map is a ring homomorphism and an isometry. 
\end{theo}
\begin{proof} The rings $SC^\bullet(V)$ and $SC^\bullet(V')$ are spanned by the classes $z_V(T,\ph)$ and $z_{V'}(T,\ph')$. 
Suppose that there is a linear relation in $SC^n(V)$
$$\sum_i c_i \,z_V(T_i,\ph_i)=0, \qquad \ph_i\in S(V(\A_f)^n), \ T_i\in \Sym_n(F), \ c_i\in \C.$$
Then, by (i) of Proposition~\ref{prop-3.6}, 
$$0 = \gs{\sum_i c_i \,z_V(T_i,\ph_i)}{z_V(T,\ph)} \ = \  \gs{\sum_i c_i \,z_{V'}(T_i,\ph'_i)}{z_{V'}(T,\ph')}$$
for all pairs $T$ and $\ph$. Here the Schwartz functions on the right side are matching those on the left via (\ref{taut-matching}). 
Since the inner product on the ring $SC^\bullet(V')$ is non-degenerate by construction, we
have
$$\sum_i c_i \,z_{V'}(T_i,\ph'_i)=0$$
in $SC^n(V')$. Thus the linear map $\rho_{V,V'}$ is well defined. 
By (ii) of Proposition~\ref{prop-3.6} this map is a ring homomorphism and by (i) it is an isometry. 
\end{proof} 

\begin{rem} (i) Of course, the comparison isomorphism of Theorem~\ref{theo3.4} is a consequence of the Siegel-Weil formula for anisotropic $V$. 
There should be an analogous comparison in the case where $d_+=d$ and $V$ is isotropic, but this will involve the 
use of a more delicate version of the (extended) Siegel-Weil formula and it seems quite likely that interesting correction 
terms will arise. \hfb
(ii) The twisted cycles $Z(T,\ph,\chibold)$ of Remarks~\ref{twisted-SCs} and \ref{rem-chi-classes} can be defined 
for both $Sh(V)_K$ and $SH(V')_K$ and one might imaging that the comparison isomorphism can be extended to these classes. 
This would require a variant of the Siegel-Weil formula giving an explicit description of the theta lifts $\theta_n(\chibold)$
from $\SO(V)$ to $\Sp_n$ 
of quadratic characters $\chibold$ of the spinor norm.  At present we do not have such a formula in general. 
The case of $\SO(3)$ and the metaplectic cover of $\SL_2$ is treated in \cite{snitz}. 
An interesting cubic analogue is considered in \cite{gan-snitz}. We plan to return to this question.
\end{rem}

\section{The case $d_+=0$}\label{section4}

In this section we suppose that $\Vp$ is a totally positive definite quadratic space of dimension $m+2$ over $F$. 
In this case, there is no associated Shimura variety, but we would like to nonetheless define a ring and to extend the comparison 
result of the previous section. 

\subsection{A ring associated to $\Vp$} 

Let $\Ze_n$ be the free abelian group on symbols $[U]_n$ where $U$ is a subspace of $\Vp$ with $\dim_F U \le n$, and let 
$$\Ze = \Ze(\Vp) = \bigoplus_{n=0}^m \Ze_n.$$
Writing $U_0 = \{0\}$ for the zero subspace of $\Vp$, we have a class $\trivn = [U_0]_0$, so that $\Ze_0 = \Z \,\trivn$, 
and a class $\cboldn = [U_0]_1\in \Ze_1$. Define a product on $\Ze$ by 
$$
[U_1]_{n_1}\cdot [U_2]_{n_2} = \begin{cases} [U_1+U_2]_{n_1+n_2} &\text{if $n_1+n_2\le m$,}\\
\nass
0&\text{otherwise.}
\end{cases}
$$
Here the `cutoff' at index $m$ is motivated by the comparison we will make below. 
Note that, $\trivn\cdot z = z$ for all $z$, and
$$ \cboldn\cdot [U]_n = \begin{cases} [U]_{n+1} &\text{if $n<m$,}\\
\nass 
0&\text{if $n=m$.}
\end{cases}
$$
Thus $\Ze$ is a graded commutative ring. 
The group $\GL(\Vp)$ acts naturally on $\Ze$ as ring automorphisms and the classes $\trivn$ and $\cboldn$ are invariant 
under this action.

For a finite subgroup $\Gamma$ in $\SO(\Vp)(F)$ and a subspace $U$ with $\dim_F U \le n$, let
$$Z_n(U)_\Gamma = \sum_{\gamma\in \Gamma/\Gamma_U} [\gamma U]_n\ \in \Ze_n,$$
where $\Gamma_U$ is the stabilizer of $U$ in $\Gamma$. 
Note that such classes span the space of $\Gamma$-invariants in $\Ze_n$. 
Then we have a `pullback' formula and a product formula, reminiscent of those for special cycles in the Shimura variety case, \cite{kudla.rem-gen}.  
\begin{lem}
(i) For a subgroup $\Gamma'\subset \Gamma$, 
$$Z_n(U)_{\Gamma} = \sum_{\gamma \in \Gamma'\back\Gamma/\Gamma_U} Z_n(\gamma U)_{\Gamma'}.$$
(ii) 
For $n_1+n_2\le m$, 
$$Z_{n_1}(U_1)_\Gamma \cdot Z_{n_2}(U_2)_{\Gamma} = \sum_{\gamma\in \Gamma\back I(U_1,U_2;\Gamma)} Z_{n_1+n_2}(W_\gamma)_\Gamma,$$
where 
$$I(U_1,U_2;\Gamma) = \Gamma/\Gamma_{U_1}\times \Gamma/\Gamma_{U_2},$$
and  \ 
$W_\gamma = \gamma_1 U_1+ \gamma_2 U_2$.
\end{lem}
Here `pullback' simply amounts to the inclusion of $(\Ze)^{\Gamma}$ in $(\Ze)^{\Gamma'}$. 

We define a degree function on $\Ze$ by 
$$ 
\deg([U]_n) = \begin{cases} 1 &\text{if $n=m$,}\\
\nass
0&\text{otherwise,}
\end{cases}
$$
and extend by linearity.  This map is invariant under the action of 
$\GL(\Vp)$.  There is a symmetric bilinear inner product on $\Ze$ defined by 
$$\gs{z_1}{z_2} = \deg(\,z_1\cdot z_2\,).$$
This inner product has a very large radical $\mathcal R \Ze$. For example, 
\beq\label{vanish-1}
\gs{[U_1]_{m-n}}{[U_2]_{n}-[U_3]_n} = 0,
\eeq
for all subspaces $U_1$ of dimension $\le m-n$ and $U_2$ and $U_3$ of dimension $\le n$. 
The radical is an ideal in $\Ze$ and we consider the quotient ring $\Zer = \Ze/\mathcal R \Ze$. 
We write $\btriv$ and $\bcbold$ for the images of the classes $\trivn$ and $\cboldn$ in $\Zer$.
\begin{lem} \label{lem4.2} The ring $\Zer$ is isomorphic to a truncated polynomial ring, 
$$\Z[y]/(y^{m+1}) \isoarrow  \Z[\bcbold] \isoarrow \Zer.$$
In particular, the image of a class $[U]_n$ in $\Zer$ is $\bcbold^{n}$. 
\end{lem}
\begin{proof}
Since $\Ze_0=\Z\,\trivn$, we have $\mathcal R\Ze \cap \Ze_m = \ker(\deg)$ and so $\Zer_0= \Z\,\triv$ and $\Zer_m = \Z \,\bcbold^m$. 
Now, by (\ref{vanish-1}), we have 
$[U_2]_n-[U_3]_n \in \mathcal R\Ze$ for all $U_2$ and $U_3$ of dimension $\le n$ and so
$$\sum_i a_i [U_i]_n \equiv (\sum_i a_i)\,(\cboldn)^n \ \mod \mathcal R\Ze.$$
This proves the claim. 
\end{proof}

\subsection{A replacement for cohomology}

For\footnote{Here we take $\SO(\Vp)$ rather than $\GSpin(\Vp)$ since there is no Shimura variety construction 
involved.}  $\Gp=R_{F/\Q}\SO(\Vp)$ and a compact open subgroup $K$ in $\Gp(\A_f)$, consider the space 
$$\HSC^\bullet(\Vp)^\nat_K:=C(\Gp(\A_f)/K,\Ze)^{\Gp(\Q)} $$ 
of functions 
$\zbold: \Gp(\A_f) \lra \Ze$ such that 
$$\zbold(\gamma g k) = \gamma\, \zbold(g),\qquad \forall \gamma\in \Gp(\Q), \ k\in K.$$
Then 
$$\HSC^\bullet(\Vp)^\nat_K = \bigoplus_{n=0}^m \HSC^n(\Vp)^\nat_K, \qquad \HSC^n(\Vp)^\nat_K=C(\Gp(\A_f)/K,\Ze_n)^{\Gp(\Q)},$$
is a graded ring.  
If we write\footnote{Note that we do not have strong approximation in this case so that the double coset space does not have a group structure. }
\beq\label{coset-decomp}
\Gp(\A_f) = \bigsqcup_j \Gp(\Q) g_j K,
\eeq
there is an isomorphism
\beq\label{Ze-iso}
\HSC^\bullet(\Vp)^\nat_K=C(\Gp(\A_f)/K,\Ze)^{\Gp(\Q)} \isoarrow \prod_j (\Ze)^{\Gamma_{j}},\qquad \zbold \mapsto [ \dots, \zbold(g_j), \dots].
\eeq
Here $\Gamma_j = \Gamma_{g_j} = \Gp(\Q)\cap g_j K g_j^{-1}$. Note that for $g\in \Gp(\A_f)$,  
the group $\Gamma_g = \Gp(\Q)\cap g K g^{-1}$ is finite and is trivial if $K$ is neat. 
Varying $K$, we have the space of continuous functions 
$$\HSC^\bullet(\Vp)^\nat= \varinjlim_K \,\HSC^\bullet(\Vp)^\nat_K = C_{\text{cont}}(\Gp(\A_f),\Ze)^{\Gp(\Q)},$$
also a ring under pointwise operations. 

Let $d^Tg$ be Tamagawa measure on $\Gp(\A)$ and define a Haar measure $d_fg$ on $\Gp(\A_f)$ via the factorization  
$d^Tg=d_\infty g_\infty\,d_fg_f$, where the archimedean factor is normalized by $\vol(\SO(\Vp)(\R),d_\infty g)=1$. 
Define a degree map on $\HSC^\bullet(\Vp)^\nat$ by\footnote{Note that the constant $\ndeg(\zbold)$ is to be distinguished from 
the function
$$\deg(\zbold) \in  C_{\text{cont}}(\Gp(\Q)\back \Gp(\A_f)).$$
}
\begin{align*}
\ndeg(\zbold) :&= \int_{\Gp(\Q)\back \Gp(\A_f)} \deg(\zbold(g))\,d_f g.
\end{align*}
For $z\in \HSC^\bullet(\Vp)^\nat_K$, by (\ref{Ze-iso}) we have 
$$
\ndeg(\zbold) =\vol(K) \sum_j |\Gamma_j|^{-1}\,\deg(\zbold(g_j)).
$$
Since the Tamagawa number of $\SO(\Vp)$ is $2$, 
$$\vol(K) \sum_j |\Gamma_j|^{-1} = \int_{\Gp(\Q)\back \Gp(\A_f)} d_f g = \int_{\Gp(\Q)\back \Gp(\A)}\,d^T g =2,$$
and we can also write
\beq\label{classical-mass}
\ndeg(\zbold) = 2\,\frac{\sum_j |\Gamma_j|^{-1}\,\deg(\zbold(g_j))}{\sum_j |\Gamma_j|^{-1}}.
\eeq
On the other hand, if $K$ is neat, we have 
$$
\HSC^\bullet(\Vp)^\nat_K \simeq \prod_j \Ze, \qquad\text{and}\qquad \ndeg(\zbold) =\vol(K) \sum_j \deg(\zbold(g_j)).
$$

We define an inner product on $\HSC^\bullet(\Vp)^\nat$ by 
\beq\label{IP-int}
\gs{\zbold_1}{\zbold_2} = \ndeg(\zbold_1 \cdot \zbold_2) =\int_{\Gp(\Q)\back \Gp(\A_f)} \deg(\zbold_1(g)\cdot \zbold_2(g))\,d_f g,
\eeq
and we let 
$$\HSC^\bullet(\Vp) = \HSC^\bullet(\Vp)^\nat/\mathcal R \HSC^\bullet(\Vp)^\nat$$
be the quotient by the radical  $\mathcal R \HSC^\bullet(\Vp)^\nat$ of this pairing. 
\begin{lem}  The ring $\HSC^\bullet(\Vp)$ is a truncated polynomial ring, 
$$C_{\text{cont}}(\Gp(\Q)\back \Gp(\A_f))\tt_\Z \Z[\cbold] \isoarrow \HSC^0(\Vp)[\bcbold] \isoarrow \HSC^\bullet(\Vp),$$
where
$$\HSC^0(\Vp) = C_{\text{cont}}(\Gp(\A_f),\Zer_0)^{\Gp(\Q)} \simeq C_{\text{cont}}(\Gp(\Q)\back \Gp(\A_f)).$$
In particular, for a class $\zbold \in \HSC^n(\Vp)$, 
\beq\label{deg-formula4.3}
\zbold = \deg(\zbold\cdot \bcbold^{m-n})\,\bcbold^n.
\eeq
\end{lem}
\begin{proof}  We claim that the subspace $C_{\text{cont}}(\Gp(\A_f),\mathcal R\Ze)^{\Gp(\Q)}$ of functions valued in the radical $\mathcal R\Ze$ 
of $\Ze$ coincides with the radical $\mathcal R\HSC^\bullet(\Vp)^\nat$. 
These functions are in the radical $\mathcal R\HSC^\bullet(\Vp)^\nat$ since, if $\zbold_1$ is such a function, the integrand 
$\deg(\zbold_1(g)\cdot \zbold_2(g))$ 
in (\ref{IP-int}) vanishes. On the other hand, if a function $\zbold$ lies in the radical  $\mathcal R \HSC^\bullet(\Vp)^\nat$, 
it is right $K$ invariant for some neat compact open $K$. For any $z'\in \Ze$, we can define a function $\zbold'_j \in  \HSC^\bullet(\Vp)^\nat_K$ 
with $\zbold'_j(g_i) = \delta_{ij} z'$. Then 
$$0=\gs{\zbold}{\zbold'_j}  = \vol(K) \,\gs{\zbold(g_j)}{z'},$$
so that $\zbold(g_j) \in \mathcal R\Ze$ for all $j$. Thus, 
$$\HSC^\bullet(\Vp) \simeq C_{\text{cont}}(\Gp(\A_f),\Zer)^{\Gp(\Q)}.$$
The lemma then follows from Lemma~\ref{lem4.2}. 
\end{proof}

\begin{rem}
Of course we could have taken 
$$\HSC^\bullet(\Vp) = C_{\text{cont}}(\Gp(\Q)\back \Gp(\A_f))[\bcbold]$$
as the definition of the `cohomology' ring in the $d_+=0$ case, but felt that the version based on $\Ze$ and $\HSC^\bullet(\Vp)^\nat$ 
provides more insight and, in particular, a better parallel with the construction in the $d_+>0$ case.  
\end{rem}

\subsection{Special cycles}\label{section4.3}
The analogue of the weighted special cycles are defined as follows. 
For $T\in \Sym_n(F)$, $n\ge1$,  and $\ph \in S(\Vp(\A_f)^n)^K$, we define
$$Z(T,\ph;K) \in \HSC^n(\Vp)^\nat_K$$ 
by 
\begin{align}\label{def-d+special}
Z(T,\ph;K)(g) &= 
 \sum_{\substack{ x\in \Vp(F)^n\\ \snass Q(x) = T\\ \snass\mod \Gamma_g}} \ph(g^{-1}x)\,Z_n(U(x))_{\Gamma_g},\notag\\
\nass
{}&= \sum_{\substack{ x\in \Vp(F)^n\\ \snass Q(x) = T}} \ph(g^{-1}x)\,[U(x)]_n,
\end{align}
where $U(x)$ is the subspace spanned by the components of $x$.  Note that the last expression is in fact independent of the 
choice of $K$, subject only to the condition that the weight function $\ph$ is $K$-invariant.  Thus we will omit $K$ from the notation.

\begin{rem} The rings $\Ze$ amd $\Zer$ as initially defined are $\Z$-algebras. Since the coefficient rings will play no 
role in our constructions, from now on we simply take complex coefficients and complex valued Schwartz functions. 
\end{rem}

Again we have a product formula. 
\begin{prop}  The product formula (\ref{prod-ch}) holds for the weighted classes,
$$
Z(T_1,\ph_1)\cdot Z(T_2,\ph_2) = \sum_{\substack{ T\in \Sym_{n_1+n_2}(F)_{\ge 0}\\ \snass T = \bpm \scr T_1&*\\{}^t*&\scr T_2\epm}} Z(T,\ph_1\tt\ph_2).
$$
\end{prop} 
\begin{proof} Writing $n=n_1+n_2$,  we have
\begin{align*}
Z(T_1,\ph_1)\cdot Z(T_2,\ph_2)(g) & = \sum_{\substack{ x_1\in \Vp(F)^{n_1}\\ \snass Q(x_1) = T_1}} \sum_{\substack{ x_2\in \Vp(F)^{n_2}\\ \snass Q(x_2) = T_2}} 
\ph_1(g^{-1}x_1)\,\ph_2(g^{-1}x_2)\,\,[U(x_1)]_{n_1}\cdot[U(x_2)]_{n_2}\\
\nass
{}&=\sum_{\substack{ T\in \Sym_{n}(F)\\ \snass T = \bpm \scr T_1&*\\{}^t*&\scr T_2\epm}} \sum_{\substack{ x= [x_1,x_2]\in \Vp(F)^{n} \\ 
\snass Q(x) = T}} \ph_1\tt\ph_2(g^{-1}x)\,[U(x_1)+U(x_2)]_n\\
\nass
{}&=\sum_{\substack{ T\in \Sym_{n}(F)\\ \snass T = \bpm \scr T_1&*\\{}^t*&\scr T_2\epm}} Z(T,\ph_1\tt \ph_2)(g).
\end{align*}
\end{proof}

By (\ref{def-d+special}) and (\ref{deg-formula4.3}), the image of the class $Z(T,\ph)$  in $\HSC^n(\Vp)$ is
\beq\label{image-ZTph}
\zbold(T,\ph)(g) = \Repp(T,\ph)(g)\cdot\bcbold^n
\eeq
where 
\beq\label{def-rep-no}
\Repp(T,\ph)(g) := \sum_{\substack{ x\in \Vp(F)^n\\ \snass Q(x) = T}} \ph(g^{-1}x)
\eeq
is the representation number.
Note that $\Repp(T,\ph)\in  C_{\text{cont}}(\Gp(\Q)\back \Gp(\A_f))$.  
The product formula in the reduced ring $\HSC^\bullet(\Vp)$ amounts to a rather trivial identity for 
such representation numbers.

We let $\SC^\bullet(\Vp)^\nat$ be the subring of $\HSC^\bullet(\Vp)$ spanned by the special cycles $\zbold(T,\ph)$  together with the class 
$\btriv$. 
Thus 
$$\SC^0(\Vp)^\nat = \C \,\btriv \ \ \subset \ \ \HSC^0(\Vp) = C_{\text{cont}}(\Gp(\Q)\back \Gp(\A_f))\,\btriv$$ 
is just the subspace of constant functions. 
Let 
$$\SC^\bullet(\Vp) = \SC^\bullet(\Vp)^\nat/\mathcal R \SC^\bullet(\Vp)^\nat$$
be the quotient of this subring by the radical of the restriction of the inner product. 
For special cycles in complementary degrees $n_1$ and $n_2$ with $n_1+n_2=m$, the inner product is given by
\beq\label{new-IP}
\gs{\zbold(T_1,\ph_1)}{\zbold(T_2,\ph_2)} = \int_{\Gp(\Q)\back \Gp(\A_f)} \Repp(T_1,\ph_1)(g) \,\Repp(T_2,\ph_2)(g) \,dg.
\eeq
We denote the image of $\zbold(T,\ph)\in \SC^\bullet(\Vp)^\nat$ in $\SC^\bullet(\Vp)$ by $\zbold(T,\ph)^\flat$. 

\begin{lem} (i) For $n\le \frac12 m$, the map 
$\SC^n(\Vp)^\nat \lra \SC^n(\Vp)$ is an isomorphism. \hfb 
(ii) 
On the other hand, for $n=m$, $\SC^m(\Vp) = \C\, \bcbold^m$ and the map 
$\SC^m(\Vp)^\nat \lra \SC^m(\Vp)$ is given by
$$\zbold \mapsto \zbold^\flat=\ndeg(\zbold)\cdot \bcbold^m.$$
\end{lem}
\begin{proof} For $n=0$ this is clear from the definition.  For $1\le n\le \frac12 m$, let $n' = m-n$. For given $T\in \Sym_n(F)$ and $\ph\in S(\Vp(\A_f)^n)$, 
let 
$$T' = \bpm T&{}\\{}&0\epm\in \Sym_{n'}(F), \qquad \ph' = \bar\ph\tt \ph^0 \in S(\Vp(\A_f)^{n'}), $$
where $\ph^0\in S(\Vp(\A_f)^{n'-n})$ with $\ph^0(0)=1$. Then, since $\Vp$ is anisotropic, 
$$\Repp(T',\ph') = \Repp(T,\bar\ph) = \overline{\Repp(T,\ph)}.$$ 
For $\zbold$ a complex linear combination of classes $\zbold(T,\ph)$'s let $\zbold'$ be the corresponding conjugate linear 
combination of the $\zbold(T',\ph')$'s, so that $\zbold'(g) = \overline{\zbold(g)}$. Then 
$$\gs{\zbold}{\zbold'} = \int_{\Gp(\Q)\back \Gp(\A_f)} |\zbold(g)|^2\,dg =0 \ \iff \zbold =0.$$
The second statement follows from the fact that 
$\gs{\zbold}{\btriv} = \ndeg(\zbold)$. 
\end{proof}

\subsection{Generating series}\label{section4.4}
The generating series analgous to (\ref{B-gen-fun}) for these `special cycles' is
\beq\label{fake-gen-fun1}
\phi_n^{\SC(\Vp)^\nat}(\tau, \ph)  = \sum_{T\in \Sym_n(F)_{\ge0}} \zbold(T,\ph)\, \qq^T \ \in \SC^n(\Vp)^\nat[[\qq]].
\eeq
The image of the series (\ref{fake-gen-fun1}) 
in $\SC^n(\Vp)[[q]]$ is 
\beq\label{fake-gen-fun2}
\phi_n^{\SC(\Vp)}(\tau, \ph)  = \sum_{T\in \Sym_n(F)_{\ge0}} \zbold(T,\ph)^\flat\, \qq^T \ \in \SC^n(\Vp)[[\qq]].
\eeq
Evaluating at $g\in \Gp(\A_f)$ 
and using (\ref{image-ZTph}) we have
\begin{align}
\label{fake-gen-fun4}
\phi_n^{\SC(\Vp)^\nat}(\tau, \ph)(g)  &=\bcbold^n\cdot \sum_{T\in \Sym_n(F)_{\ge0}} \Repp(T,\ph)(g)\,\qq^T\notag\\
\nass
{}&=  \bcbold^n\cdot \sum_{\substack{ x\in \Vp(F)^n}} \ph(g^{-1}x)\,\qq^{Q(x)} \\
\nass
{}&=\bcbold^n\cdot N(\det(v))^{-\frac{m+2}4} \,\theta(g'_\tau,g;\phbold^{(n)}_\infty\tt\ph),\notag
\end{align}
where the function in the last line is a multiple of the classical theta series (\ref{J-theta-form}) given by 
$$ 
\theta(g',g;\phbold^{(n)}_\infty\tt\ph) = \sum_{x\in \Vp(F)^n} \o(g')(\phbold^{(n)}_\infty\tt\ph)(g^{-1}x),
$$ 
for
\beq\label{arch-SF-2}
\phbold^{(n)}_\infty= \breve\phbold^{(n)}_\infty=\bigotimes_{\s} \ph_{\s,+}^0 
\eeq
where $\ph_{\s,+}^0\in S((\Vp)_\s^n)$  is the Gaussian for $(V_+)_\s$.

Now we have the analogues of the formulas of Section~\ref{section2}.  First we have the analogue of Proposition~\ref{prop2.2}, where in the 
third step we use (\ref{fake-gen-fun4}), 
\begin{align*}
\ndeg(\phi_n^{\SC(\Vp)}(\tau,\ph)\cdot \bcbold^{m-n}) & = 
\int_{\Gp(\Q)\back \Gp(\A_f)}
	\deg(\phi_n^{\SC(\Vp)^\nat}(\tau,\ph)(g)\cdot \bcbold^{m-n})\,d_fg\\
\nass
{}&=\int_{\Gp(\Q)\back \Gp(\A_f)}
	N(\det(v))^{-\frac{m+2}4} \,\theta(g'_\tau,g;\phbold^{(n)}_\infty\tt\ph)\,d_fg\\
\nass
{}&=N(\det(v))^{-\frac{m+2}4} \, \int_{\Gp(\Q)\back \Gp(\A)}
	\theta(g'_\tau,g;\phbold^{(n)}_\infty\tt\ph)\,d^Tg\\
\nass
{}&= 2\,E(\tau,s_0,\l_{(\Vp)_f}(\ph)).
\end{align*}
Here in the last step we use the Siegel-Weil formula and the expression in (\ref{good-Eis}) for the special value at $s=s_0 = \kappa-\rho_n$ of the Siegel-Eisenstein series. 

If $\ph$ is $K$-invariant for a compact open subgroup $K$ in $\Gp(\A_f)$, then, using (\ref{coset-decomp}),  (\ref{classical-mass})
and  (\ref{Eis-FC}), we have
Siegel's classical formula
\beq\label{Siegel!}
\ndeg(\zbold(T,\ph)^\flat\cdot \bcbold^{m-n}) = 2\,\frac{\sum_j |\Gamma_j|^{-1}\,\Repp(T,\ph)(g_j)}{\sum_j |\Gamma_j|^{-1}} = 2\,A(T,\l_{V_f}(\ph))
\eeq
relating representation numbers and Fourier coefficients of Eisenstein series. This is the analogue of Corollary~\ref{cor2.2}.

The analogues of Theorem~A and Corollary~\ref{IP-cor} follow in the same way.  We will not restate them here. 
They imply that the product structure and inner product of classes $\zbold(T,\ph)^\flat$ in the ring $\SC^\bullet(\Vp)$ are once again 
given by Fourier coefficients of pullbacks of Hilbert-Siegel Eisenstein series.  
The product structure in the quotient ring is now much more subtle than that in $\SC^\bullet(\Vp)^\nat$, as it involves the inner products (\ref{new-IP}) 
of the functions $\Repp(T,\ph)$ on $\Gp(\Q)\back \Gp(\A_f)$.  For example, 
if  $m=2n$ is even and $T_i\in \Sym_n(F)$ and $\ph_i\in S(\Vp(\A_f)^n)$, then 
$$\zbold(T_1,\ph_1)^\flat\cdot \zbold(T_2,\ph_2)^\flat = \int_{\Gp(\Q)\back \Gp(\A_f)} \Repp(T_1,\ph_1)(g) \,\Repp(T_2,\ph_2)(g)\,dg \cdot \bcbold^m.$$

\subsection{Another comparison}  
The results of Sections~\ref{section4.3} and \ref{section4.4} imply that we again have a comparison isomorphism.
\begin{theo}\label{theo4.6}  For a quadratic space $V$ over $F$ with $d_+(V)$ even, let $V_+$ be the associated totally positive definite space. 
Fix an isometry $\rho_{V,V_+}: V(\A_f) \isoarrow V_+(\A_f)$. 
Let $\SC^\bullet(V_+)$ be the `special cycle' ring for $V_+$ defined in Section~\ref{section4.3}, and let $\SC^\bullet(V)$ be the 
reduced special cycle ring defined in Section~\ref{subsec-1.2}.
Then there is a linear map
$$\rho_{V,V_+}: \SC^\bullet(V) \lra \SC^\bullet(V_+)$$
such that, for $\ph$ and $\ph'$ matching as in (\ref{taut-matching}), 
$$\rho_{V,V_+}:  z_V(T,\ph) \mapsto \zbold_{V_+}(T,\ph')^\flat.$$
Moreover, this map is a ring homomorphism and an isometry. 
\end{theo}

Here we have added a subscript to indicate where the classes live.

Thus, the reduced special cycle rings $\SC^\bullet(V)_K$ for the Shimura varieties $Sh(V)_K$ for such $V$ with $d_+(V)$ even are all modeled on
the subquotient $\SC^\bullet(\Vp)$ of the truncated polynomial ring
\begin{align*}
\HSC^\bullet(\Vp)_{\bar K}&\isoarrow C_{\text{cont}}(\Gp(\Q)\back \Gp(\A_f)/\bar K)[\bcbold].
\end{align*}
Here our choice of isometry $\rho_{V,V_+}: V(\A_f) \isoarrow V_+(\A_f)$ gives an identification 
$G(\A_f) \isoarrow \GSpin(V_+)(\A_f)$, and we write $\bar K$ for the image of $K\subset G(\A_f)$ in $G_+(\A_f)= \SO(V)(\A_f)$.

\section{Local matching conditions}\label{section5}

It is interesting to see how much matching can occur in cases where the spaces $V$ and $V'$ are not locally isometric at all places. 
For $V$ and $V'$ as in Proposition~\ref{prop3.5},  there is a finite set of places 
$$\Delta = \Delta(V,V') = \{ \,\pp \mid V_\pp \not\simeq V'_\pp\,\} = \{ \,\pp \mid \e(V_\pp) = - \e(V'_\pp)\,\},$$
where the Hasse invariants $\e(V_\pp)$ and $\e(V'_\pp)$ differ.  
We fix an isomorphism 
\beq\label{almost.matching}
V(\A_f^\Delta) \isoarrow V'(\A_f^\Delta),
\eeq
where the superscript $\Delta$ means that the places in $\Delta$ have been omitted. 
This gives isomorphisms
$$S(V(\A_f^\Delta)^n) \isoarrow S(V'(\A_f^\Delta)^n), \qquad \ph \mapsto \ph'$$
for all $n$, compatible with  the Weil representation and with tensor products. In particular, the functions $\ph$ and $\ph'$ are matching.  
The existence of matching pairs of functions in $S(V(\A_f)^n)$ and $S(V'(\A_f)^n)$ then reduces to a local problem at places $\pp\in  \Delta$. 

\subsection{Local matching conditions}
In this section, we describe the conditions required for local matching using the results of \cite{KR.rdps}, for $m$ even, 
and \cite{sweet.meta}, for $m$ odd. We change notation and let $F$ be a non-archimedean local field of characteristic $0$
with a fixed non-trivial unitary additive character $\psi: F\ra \C^\times$.  For a fixed $m\ge 1$, let 
$$G' =G'_n= \begin{cases} \Sp_n(\F) &\text{if $m$ is even} \\
\nass
\nass
\widetilde{\Sp_n}(F)&\text{its metaplectic cover, if $m$ is odd.}
\end{cases}
$$ 
For a quadratic space $V$ over $F$ of dimension\footnote{Note the shift in notation from the previous sections.} $m$ and character $\chi_V$,  let  $\o_V=\o_{V,\psi}$ be the Weil representation of $G'$ 
on $S(V^n)$. For the intertwining map 
$$\l_V: S(V^n) \lra I_n(s_0,\chi_V), \qquad \l_V(\ph)(g') = \o(g')\ph(0), \quad s_0= \frac{m}2-\rho_n,$$
to the degenerate principal series representation $I_n(s,\chi_V)$ of $G'$ at the point $s_0$,  
we let 
$$R_n(V) = \l_V(S(V^n)) \  \subset I_n(s_0,\chi_V)$$
be the image. 
For a fixed quadratic character $\chi$ of $F^\times$, these submodules account for the constituents of $I_n(s_0,\chi)$. 
In the following description of the $R_n(V)$'s, for convenient reference, we give more complete information than needed for our application to matching. 
For $m$ even, these results are quoted from \cite{KR.rdps}, while, $m$ odd the results are due to Sweet, \cite{sweet.meta}. 

We consider quadratic spaces $V$ with $\chi_V=\chi$ and we vary $\dim V=m$. For a fixed $m$, 
the isometry class of such a space $V$ is determined by its Hasse invariant $\e(V)=\pm1$. 
In particular, up to isometry,  there are two such spaces $V_{m,\pm}$, except 
when $m=1$ (resp. $m=2$ and $\chi=1$), in which case there is only one such space $V_1$ (resp. $V_2$).  
The `generic' picture of the $R_n(V)$'s is quite simple. 
\begin{prop}
\begin{itemize}
\item[(i)]  For $3\le m < n+1$, or for $m=2$, $\chi\ne1$, we have $s_0<0$, and the representations $R_n(V_{m,\pm})$ are irreducible and distinct. 
Their sum  $R_n(V_{m,+})\oplus R_n(V_{m,-})$ 
is the socle of $I_n(s_0,\chi)$ and the quotient $I_n(s_0,\chi)/R_n(V_{m,+})\oplus R_n(V_{m,-})$ is irreducible. 
\item[(ii)] For $m=n+1$, we have $s_0=0$. If $n=1$  and $\chi=1$, $R_1(V_2)=I_1(0,\chi)$ is irreducible. 
Otherwise,  the spaces $R_n(V_{n+1,\pm})$ are irreducible and distinct
and 
$$I_n(0,\chi) = R_n(V_{n+1,+})\oplus R_n(V_{n+1,-}).$$ 
\item[(iii)] For $n+1<m<2n$,   
the spaces $R_n(V_{m,\pm})$ are maximal subspaces of $I_n(s_0,\chi)$,  $I_n(s_0,\chi) = R_n(V_{m,+})+R_n(V_{m,-})$, $R_n(V_{m,+})\cap R_n(V_{m,-})$ is irreducible, and 
$$I_n(s_0,\chi)/R_n(V_{m,+})\cap R_n(V_{m,-}) \isoarrow R_n(V_{m',+})\oplus R_n(V_{m',-}),$$
where $m'=2n+2-m$ and the isomorphism is induced by the normalized intertwining map $A_n(s_0,\chi)$. \hfb 
The same statement holds for $m=2n$ when  $\chi\ne1$. 
\item[(iv)] For $m>2n+2$, or if $m=2n+2$ and $\chi\ne1$, $I_n(s_0,\chi) = R_n(V_{m,\pm})$ is irreducible. 
\end{itemize}
The following `edge' cases then complete the picture. 
\begin{itemize}
\item[(a)]  For $m=1$, we have $s_0=-\frac{n}2$. Then  $R_n(V_1)$ is the unique irreducible submodule of $I_n(-\frac{n}2,\chi)$, and 
the quotient 
$$I_n(-\frac{n}2,\chi)/R_n(V_1)\isoarrow R_n(V_{2n+1,-})$$
is irreducible. 
\item[(b)] For $m=2$,  $\chi=1$, and $n>1$, we have $s_0= -\rho_n+1$. Then $R_n(V_2)$ is the unique irreducible submodule of $I_n(-\frac{n}2+\frac12,\chi)$ and the quotient
$$I_n(-\frac{n}2+\frac12,\chi)/R_n(V_2) \isoarrow R_n(V_{2n,-})$$
is irreducible. 
\item[(c)] For $m=2n$ and $\chi=1$, we have $s_0= \frac{n-1}2$. Then  $I_n(\frac{n-1}2,\chi) = R_n(V_{2n,+})$, $R_n(V_{2n,-})$ is its unique irreducible submodule, and
$$I_n(\frac{n-1}2,\chi)/R_n(V_{2n,-}) \isoarrow R_n(V_2).$$
\item[(d)] For $m=2n+1$, we have $s_0= \frac{n}2$. Then  $I_n(\frac{n}2,\chi) = R_n(V_{2n+1,+})$, $R_n(V_{2n+1,-})$ is its unique irreducible submodule, and 
$$I_n(\frac{n}2,\chi)/R_n(V_{2n+1,-}) \isoarrow R_n(V_1).$$
\item[(e)] For $m=2n+2$ and $\chi=1$, we have $s_0=\rho_n$. Then $I_n(\rho_n,\chi) = R_n(V_{2n+2,+})$, 
$R_n(V_{2n+2,-})$ is its unique irreducible submodule, and
$$I_n(\rho_n,\chi)/R_n(V_{2n+2,-}) \isoarrow R_n(V_0),$$
where we formally view the trivial representation as $\triv = R_n(V_0)$. 
\end{itemize} 
Moreover, each of the quotients occurring in cases (b)--(e) is induced by the normalized intertwining operator $A_n(s_0,\chi)$, \cite{KR.rdps}, 
\cite{sweet.meta}. 
\end{prop}
Note that $R_n(V_1)$ is the even Weil representation of $G'_n$. 

Now return to the local matching problem and suppose that $V$ and $V'$ are quadratic spaces over $F$ of dimensions $m=\mbold+2$
and character $\chi$, but with opposite Hasse invariants. 
We now vary $n$ with $1\le n \le \mbold = m-2$. In particular, $m\ge 3$. 
For $\e=\pm1$, let
$$S^o(V_{m,\e}^n) = \{\,\ph \in S(V_{m,\e}^n)\mid \l_{V_{m,\e}}(\ph)\in R_n(V_{m,-\e}) \, \},$$
so that, tautologically, 
for every $\ph\in S^o(V_{m,\pm}^n)$ there is a matching function $\ph'\in S^o(V_{m,\mp}^n)$. 
Also note that, since $R_n(V_{m,-\e})$ is a $G'_n$ submodule of $I_n(s_0,\chi)$ and $\l_V$ is intertwining, 
$S^o(V_{m,\e}^n)$ is a $G'_n$-invariant subspace of $S(V_{m,\e}^n)$. 

\begin{prop}\label{prop3.7}
{\rm (1)} For $n<\frac{m}2-1$, or $n=\frac{m}2-1$ and $\chi\ne1$, $S^o(V_{m,\e}^n) = S(V_{m,\e}^n)$, and hence 
for every $\ph\in S(V_{m,\pm}^n)$ there is a matching function $\ph'\in S(V_{m,\mp}^n)$.
\hfb
{\rm (2)} For $n=\frac{m}2-1$ and $\chi=1$, $s_0=\rho_n$. Then, 
$$S^o(V_{m,-}^n)  = S(V_{m,-}^n), \qquad S^o(V_{m,+ }^n) = \ker(A_n(\rho_n,\chi)\circ\l_V),$$
and 
$$S(V_{m,+ }^n) /S^o(V_{m,+ }^n)  \isoarrow R_n(V_0)=\triv.$$
Similarly, 
for $n=\frac{m}2$ and $\chi=1$, $s_0=\rho_n-1$. Then, 
$$S^o(V_{m,-}^n)  = S(V_{m,-}^n), \qquad S^o(V_{m,+ }^n) = \ker(A_n(\rho_n-1,\chi)\circ\l_V),$$
and 
$$S(V_{m,+ }^n) /S^o(V_{m,+ }^n)  \isoarrow R_n(V_2).$$
{\rm (3)} 
For $n=\frac{m}2-\frac12$ so that  $s_0=\rho_n-\frac12=\frac{n}2$, 
$$S^o(V_{m,-}^n)  = S(V_{m,-}^n), \qquad S^o(V_{m,+ }^n) = \ker(A_n(\rho_n-\frac{1}2,\chi)\circ\l_V)$$
and 
$$S(V_{m,+ }^n) /S^o(V_{m,+ }^n)  \isoarrow R_n(V_1).$$
{\rm (4)} For $\frac{m}2<n\le m-2$, or for $n=\frac{m}2$ and $\chi\ne1$, there are exact sequences induced by the normalized intertwining
operator $A_n(s_0,\chi)$, 
$$0\lra S^o(V_{m,\pm}^n) \lra S(V_{m,\pm}^n) \lra R_n(V_{m',\pm}) \lra 0,\qquad m'=2n+2-m.$$
\end{prop}

\begin{rem}\label{rem5.3} (i)
For a fixed $n$, Proposition~\ref{prop3.7}, together with the isomorphism (\ref{almost.matching}), provides a good supply of matching pairs to which Proposition~\ref{prop3.5} can be applied. \hfb
(ii)
Comparison of inner products is more difficult to achieve, and it is not clear if one can expect to find isomorphisms like that of Theorem~\ref{theo3.4}. 
\hfb
(iii)
For example, suppose that $m=\mbold +2$ is even and $\chi\ne1$. Take $n = \frac12\mbold$, i.e., $m=2n+2$.  Then by (iv), $R_n(V_{m,\pm}) = I_n(s_0,\chi)$
and hence every function $\ph \in S(V^n_{m,+})$ has a matching function $\ph'\in S(V^n_{m,-})$. 
To compare an inner product of classes as in (i) of Proposition~\ref{prop-3.6}, we would want to have $\ph_1\tt\ph_2$ matching $\ph_1'\tt\ph_2'$ as well. 
On the other hand, we have the sequence of surjective maps
$$S(V^n_{m,\pm})\tt S(V^n_{m,\pm}) = S(V^{\mbold}_{m,\pm}) \lra R_{\mbold}(V_{m,\pm}) \lra R_{\mbold}(V_{\mbold,\pm}),$$
so that there can be no matching function when the image of $\ph_1\tt\ph_2$ in $R_{\mbold}(V_{\mbold,\pm})$
is non-zero.  This produces a supply of examples where the tensor products of matching functions are not matching.

\end{rem}

\end{document}

%% file: KRY.macros.tex



\newcommand{\A}{{\mathbb A}}
\newcommand{\C}{{\mathbb C}}
\newcommand{\E}{{\mathbb E}}
\newcommand{\F}{{\mathbb F}}
\newcommand{\G}{{\mathbb G}}
\newcommand{\R}{{\mathbb R}}
\newcommand{\Q}{{\mathbb Q}}
\newcommand{\X}{{\mathbb X}}
\newcommand{\Z}{{\mathbb Z}}
\newcommand{\HZ}{\widehat{\Z}}
\newcommand{\Sp}{\hbox{Sp}}

\newcommand{\rom}[1]{\text{\rm #1}}
\renewcommand{\roman}{\rm}

\newcommand{\Aut}{\text{\rm Aut}}
\newcommand{\CH}{\widehat{\text{\rm CH}}}
\newcommand{\cha}{{\text{\rm char}}}
\newcommand{\CHe}{\text{\rm CHeeg}}
\newcommand{\degh}{\widehat{\text{\rm deg}}}
\newcommand{\degH}{\widehat{\text{\rm deg}}}    
\newcommand{\diag}{{\text{\rm diag}}}
\newcommand{\Diff}{\text{\rm Diff}}
\newcommand{\disc}{\text{\rm discr}}
\renewcommand{\div}{\text{\rm div}}
\newcommand{\divh}{\widehat{\text{\rm div}}}
\newcommand{\DS}{\text{\rm DS}}
\newcommand{\Ei}{\text{\rm Ei}}
\newcommand{\End}{\text{\rm End}}
\newcommand{\ev}{{\text{\rm ev}}}
\newcommand{\Gal}{\text{\rm Gal}}
\newcommand{\GL}{\text{\rm GL}}
\newcommand{\GSpin}{\text{\rm GSpin}}
\newcommand{\Hom}{\text{\rm Hom}}
\newcommand{\hor}{{\text{\rm horiz}}}
\newcommand{\id}{\text{\rm id}}
\newcommand{\im}{\text{\rm im}}
\renewcommand{\Im}{\text{\rm Im}}
\newcommand{\inv}{{\text{\rm inv}}}
\newcommand{\Jac}{\text{\rm Jac}}
\newcommand{\Leray}{{\mathrm L}}
\newcommand{\Lie}{\text{\rm Lie}}
\newcommand{\Mp}{\text{\rm Mp}}
\newcommand{\mult}{\text{\rm mult}}
\newcommand{\MW}{\text{\rm MW}}
\newcommand{\MWt}{\widetilde{\MW}}
\newcommand{\new}{\text{\rm new}}
\newcommand{\Nm}{\text{\rm Nm}}
\newcommand{\ord}{\text{\rm ord}}
\newcommand{\PGL}{\text{\rm PGL}}
\newcommand{\Pic}{\text{\rm Pic}}
\newcommand{\Pich}{\widehat{\text{\rm Pic}}}
\newcommand{\pr}{\text{\rm pr}}
\newcommand{\ra}{\text{\rm ra}}
\newcommand{\Rao}{\mathrm R}
\renewcommand{\Re}{\text{\rm Re}}
\newcommand{\sgn}{\text{\rm sgn}}
\newcommand{\sig}{\text{\rm sig}}
\newcommand{\SL}{\text{\rm SL}}
\newcommand{\SO}{\text{\rm SO}}
\renewcommand{\Sp}{\text{\rm Sp}}
\newcommand{\Spec}{\text{\rm Spec}\, }
\newcommand{\Spf}{\text{\rm Spf}}
\newcommand{\supp}{\text{\rm supp}}
\newcommand{\Sym}{{\text{\rm Sym}}}
\newcommand{\tr}{\text{\rm tr}}
\newcommand{\type}{\text{\rm type}}
\newcommand{\Ver}{\text{\rm Vert}}
\newcommand{\vol}{\text{\rm vol}}
\newcommand{\Wald}{\text{\rm Wald}}


\newcommand{\Cal}{\mathcal}     

\newcommand{\AHH}{\hat{\Cal A}}   
\newcommand{\CHH}{\hat{\Cal C}}
\newcommand{\MM}{\Cal D}          
\newcommand{\MMb}{\MM^\bullet}
\newcommand{\ssplit}{\text{\bf split}}
\newcommand{\whcc}{\widehat{\Cal C}}
\newcommand{\CO}{\mathcal O}
\newcommand{\COH}{\widehat{\CO}}
\newcommand{\M}{\Cal M}
\newcommand{\OB}{\Cal O_B}
\newcommand{\XX}{\mathcal X}
\newcommand{\bXX}{\bar\XX}
\newcommand{\wc}{\hat{\Cal C}}
\newcommand{\wch}{\wc^{\text{\rm hor}}}
\newcommand{\ZZ}{\Cal Z}
\newcommand{\ZH}{\widehat{\Cal Z}}   
\newcommand{\Zh}{\widehat{\Cal Z}}
\newcommand{\ZZh}{\ZZ^{\text{\rm hor}}}
\newcommand{\ZZv}{\ZZ^{\text{\rm ver}}}
\newcommand{\ZZhh}{\Zh^{\text{\rm hor}}}
\newcommand{\ZZhv}{\Zh^{\text{\rm ver}}}


\newcommand{\nass}{\noalign{\smallskip}}
\newcommand{\snass}{\noalign{\vskip 2pt}}
\newcommand{\tent}[1]{ \vphantom{\vbox to #1pt{}} }   


\newcommand{\scr}{\scriptstyle}
\newcommand{\disp}{\displaystyle}

\font\cute=cmitt10 at 12pt
\font\smallcute=cmitt10 at 9pt
\newcommand{\kay}{{\text{\cute k}}}
\newcommand{\smallkay}{{\text{\smallcute k}}}

\renewcommand{\a}{\alpha}
\renewcommand{\b}{\beta}
\newcommand{\e}{\epsilon}
\renewcommand{\l}{\lambda}
\renewcommand{\L}{\Lambda}
\renewcommand{\o}{\omega}
\renewcommand{\O}{\Omega}
\renewcommand{\P}{\Phi}
\newcommand{\ph}{\varphi}
\newcommand{\phih}{\widehat{\phi}}
\newcommand{\wphi}{\widehat{\phi}}
\newcommand{\phit}{\widetilde{\phi}}
\newcommand{\s}{\sigma}
\newcommand{\vth}{\vartheta}


%

\newcommand{\Pt}{P}
\newcommand{\Ph}{\P}
\newcommand{\Pht}{\tilde \P}   
\newcommand{\Kt}{K}           
\newcommand{\Mt}{M}

\newcommand{\pht}{\widetilde{\phi}}
\newcommand{\It}{I}
\newcommand{\Jt}{\widetilde{J}}
\newcommand{\lt}{\widetilde{\l}}
\newcommand{\vp}{\varpi}

\newcommand{\bom}{{\boldsymbol{\o}}}
\newcommand{\hbom}{\widehat{\bom}}
\newcommand{\ff}{{\bold f}}
\newcommand{\fsp}{\boldsymbol{f}_{\!\rm sp}}
\newcommand{\fev}{\boldsymbol{f}_{\!\rm ev}}
\newcommand{\fb}{\boldsymbol{f}}
\newcommand{\J}{\und{J}'}
\newcommand{\JJ}{\bold J'}
\newcommand{\V}{\bold V}
\newcommand{\xx}{\bold x}

\newcommand{\g}{{\mathfrak g}}
\renewcommand{\H}{\mathfrak H}


\newcommand{\back}{\backslash}
\newcommand{\CT}[1]{\operatornamewithlimits{CT}_{#1}}
\renewcommand{\d}{\partial}
\newcommand{\db}{\bar\partial}
\newcommand{\dbar}{\bar{\partial}}
\newcommand{\gs}[2]{\langle \,#1,#2\,\rangle}
\newcommand{\Gt}{G}
\newcommand{\hfal}{h_{\text{\rm Fal}}}
\newcommand{\II}{\int^\bullet}
\newcommand{\isoarrow}{\ {\overset{\sim}{\longrightarrow}}\ }
\newcommand{\lisoarrow}{\ {\overset{\sim}{\longleftarrow}}\ }
\newcommand{\limdir}[1]{\underset{\underset{#1}{\rightarrow}}{\lim}}
\newcommand{\lan}{\operatorname{\langle}\hskip .5pt}
\newcommand{\ran}{\,\operatorname{\rangle}}
\newcommand{\lra}{\longrightarrow}
\newcommand{\doublelra}{\ {\overset{\scr\lra}{\scr\lra}}\ }
\newcommand{\nat}{\natural}
\newcommand{\notmid}{\mkern-5mu\not\mkern5mu\mid}
\newcommand{\Optoc}{\text{\rm Opt}(O_{c^2d},O_B)}
\newcommand{\psim}{\psi^{-}}
\newcommand{\qeq}{\ \overset{??}{=}\ }
\newcommand{\sh}{\sharp}
\newcommand{\thCH}{\theta^{\text{\rm ar}}}
\newcommand{\wht}{\widehat{\theta}}     
\newcommand{\triv}{1\!\!1}
\renewcommand{\tt}{\otimes}
\newcommand{\und}[1]{\underline{#1}}
\newcommand{\z}{z}  

\newcommand{\thMW}{\theta^{\text{\rm ar}}}
\newcommand{\tph}{\widetilde{\widehat\phi_1}}
\newcommand{\Pet}{\text{\rm Pet}}





\newcommand{\thing}{ \raisebox{-6.4pt}{$\tilde{\tilde{}}$}  }   
\newcommand{\smallthing}{ \raisebox{-4.4pt}{$\scr\tilde{\tilde{}}$}  }
\newcommand{\ttilde}[1]{\overset{\smash{\thing}}{#1}}
\newcommand{\smallttilde}[1]{\overset{\smash{\smallthing}}{#1}}
\newcommand{\downhookarrow}{\hbox{$\downarrow\hskip -6.1pt\raisebox{6pt}{$\cap$}$}}


\newcommand{\hfb}{\hfill\break}
\newcommand{\margincom}[1]{\marginpar{\bf\raggedright #1}}
\newcommand{\Sec}{\S}


\numberwithin{equation}{section}
\setcounter{section}{0}
\setcounter{MaxMatrixCols}{15}


\newtheorem{theo}{Theorem}[section]
\newtheorem{lem}[theo]{Lemma}
\newtheorem{prop}[theo]{Proposition}
\newtheorem{cor}[theo]{Corollary}
\newtheorem*{main}{Main Theorem}
\newtheorem*{maina}{Main Theorem A}
\newtheorem*{mainb}{Main Theorem B}
\newtheorem*{atheo}{Theorem A}
\newtheorem{conj}[theo]{Conjecture}
\theoremstyle{remark}
\newtheorem{rem}[theo]{Remark}      
\newtheorem{example}[theo]{Example}
\theoremstyle{definition}
\newtheorem{defn}[theo]{Definition}